\documentclass[onefignum,onetabnum]{siamart190516}

\ifpdf
  \DeclareGraphicsExtensions{.eps,.pdf,.png,.jpg}
\else
  \DeclareGraphicsExtensions{.eps}
\fi

\newsiamremark{remark}{Remark}
\newsiamremark{example}{Example}
\newsiamremark{hypothesis}{Hypothesis}
\crefname{hypothesis}{Hypothesis}{Hypotheses}
\newsiamthm{claim}{Claim}
\newsiamthm{assumption}{Assumption}
\crefname{assumption}{Assumption}{Assumptions}

\headers{Strong Stationarity for a Rate-Independent EVI}
{Martin Brokate and Constantin Christof}

\title{Strong Stationarity Conditions for Optimal Control Problems Governed by 
a Rate-Independent Evolution Variational Inequality\thanks{Submitted to the editors DATE.
}}

\author{Martin Brokate\thanks{%
Technische Universit\"at M\"unchen,
Department of Mathematics, M6, 
Boltzmannstra{\ss}e 3,
85748 Garching bei M\"unchen, Germany;
Weierstrass Institute for Applied Analysis and Stochastics, 
Mohrenstra{\ss}e 39, 
10117 Berlin, Germany;
Faculty of Civil Engineering, Czech Technical University in Prague, 
Th\'{a}kurova 7, 16629 Praha 6, Czech Republic;
\url{https://www.professoren.tum.de/brokate-martin},
\email{brokate@ma.tum.de}%
}%
\and
Constantin Christof\thanks{%
Technische Universit\"at M\"unchen,
Department of Mathematics, M17,
Boltzmannstra{\ss}e 3,
85748 Garching bei M\"unchen, Germany,
\url{https://www-m17.ma.tum.de/Lehrstuhl/ConstantinChristof},
\email{christof@ma.tum.de}
}%
}

\usepackage{amsopn}

\usepackage{amsfonts}
\usepackage{graphicx}
\usepackage{amssymb}
\usepackage{enumitem}
\usepackage{dsfont}

\newcommand{\sgn}{\operatorname{sgn}}
\newcommand{\closure}{\operatorname{cl}}
\newcommand{\var}{\operatorname{var}}
\newcommand{\dd}{\,\mathrm{d}}
\newcommand{\crit}{\mathrm{crit}}
\newcommand{\red}{\mathrm{red}}
\newcommand{\rad}{\mathrm{rad}}
\newcommand{\supp}{\operatorname{supp}}
\newcommand{\ptw}{\mathrm{ptw}}
\newcommand{\Uad}{U_{\textup{ad}}}

\newcommand{\II}{\mathcal{I}}
\newcommand{\JJ}{\mathcal{J}}
\newcommand{\KK}{\mathcal{K}}

\newcommand{\PP}{\mathcal{P}}

\renewcommand{\S}{\mathcal{S}}

\newcommand{\R}{\mathbb{R}}

\newcommand{\ddt}{\frac{\mathrm{d}}{\mathrm{dt}}}

\newcommand{\weakly}{\rightharpoonup}
\newcommand{\weaklystar}{\stackrel\star\rightharpoonup}

\definecolor{darkgreen}{rgb}{0,0.5,0}
\definecolor{darkred}{rgb}{0.8,0,0}

\usepackage{marginnote}
\newcommand{\stkout}[1]{\ifmmode\text{\sout{\ensuremath{#1}}}\else\sout{#1}\fi}

\usepackage{etoolbox}
\makeatletter
\patchcmd{\@addmarginpar}{\ifodd\c@page}{\ifodd\c@page\@tempcnta\m@ne}{}{}
\makeatother
\reversemarginpar

\begin{document}

\maketitle

\begin{abstract}%
We prove strong stationarity conditions 
for optimal control problems that are  governed
by a prototypical  
rate-independent evolution variational inequality,
i.e., first-order necessary optimality conditions in the form of a 
primal-dual multiplier system that are
equivalent to the purely primal notion of Bouligand stationarity. 
Our analysis relies on recent results on the 
Hadamard directional differentiability of the scalar stop
operator and a new concept of temporal polyhedricity that generalizes
classical ideas of Mignot.
The established strong stationarity system is compared with 
known optimality conditions for 
optimal control problems governed by 
elliptic obstacle-type variational inequalities and 
stationarity systems obtained by regularization.
\end{abstract}



\section{Introduction and summary of results}%
\label{sec:1}%
This paper is concerned with the derivation of  
first-order
necessary optimality conditions for
optimal control problems of the type
\begin{equation*}
\label{eq:P}
\tag{P}
\left \{~~
	\begin{aligned}
		\text{Minimize} 
		\quad &  \JJ(y, y(T), u)  \\
        \text{w.r.t.}
        \quad &y \in CBV[0, T], \quad  u \in \Uad,\\
		\text{s.t.}
		\quad & \int_0^T (v - y)\dd (y - u)\geq 0~~ \forall v \in C([0, T]; Z), \\
        &y(t) \in Z\quad \forall t \in [0, T],\quad y(0) = y_0.
	\end{aligned}
\right.
\end{equation*}
Here, 
$y$ denotes the state; 
$u$ denotes the control; 
$T>0$ is given;
$CBV[0,T]$ is the space of real-valued 
continuous functions of bounded variation on $[0, T]$; 
$\Uad$ is a subset of a suitable control space $U \subset CBV[0,T]$;
$\JJ\colon L^\infty(0, T) \times \R \times U \to \R$ 
is a sufficiently smooth objective function;
$Z = [-r,r]$ is a given interval with $r>0$; 
$C([0, T]; Z)$ is the set of continuous functions 
on $[0,T]$ with values in $Z$; 
$y_0 \in Z$ is a given initial value; 
and the integral in the governing 
variational inequality is understood 
in the sense of Kurzweil-Stieltjes (see \cite{Monteiro2019} and the  \hyperref[sec:appendix]{appendix} of this paper
for details on this type of integral).
For the precise assumptions on the quantities in 
\eqref{eq:P}, we refer to \cref{sec:3}. 
The main result of this work -- \cref{th:main} -- establishes 
a so-called strong stationarity system
for the problem \eqref{eq:P}. 
This is a first-order necessary optimality condition in primal-dual form
that is satisfied by a control $\bar u \in \Uad$
if and only if $\bar u$ is a Bouligand stationary point
of \eqref{eq:P}, i.e., 
if and only if the directional derivative of 
the reduced objective function of \eqref{eq:P}
at $\bar u$
is nonnegative in all admissible directions.
See also \eqref{eq:strongstatsys-2} below for the resulting 
stationarity system. 

\subsection{Background and relation to prior work}
Before we present and discuss the 
strong stationarity system derived in \cref{th:main}
in more detail, let us give some background. 
To keep the discussion concise, we 
focus on strong stationarity conditions 
for infinite-dimensional optimization problems 
arising in optimal control. 
For related results 
in finite dimensions, 
see \cite{Flegel2007,Harder2017,Hoheisel2013,Luo1996,Scheel2000}
and the references therein. 

In the field of infinite-dimensional nonsmooth optimization, 
strong stationarity conditions
(although originally not referred to as such)
have first been derived for optimal control problems 
governed by elliptic obstacle-type variational inequalities 
in the seminal works \cite{Mignot1976,MignotPuel1984}
of Mignot and Puel in the nineteen-seventies and -eighties. 
If we use a notation analogous to that 
in \eqref{eq:P}, then this kind of problem 
can be formulated (in its most primitive form)
as follows:
\begin{equation}
\label{eq:optstaclecontrol}
	\begin{aligned}
		\text{Minimize} 
		\quad &  \JJ(y, u)  \\
        \text{w.r.t.}
        \quad &y \in H_0^1(\Omega), \quad  u \in \Uad \subset L^2(\Omega),\\
		\text{s.t.}
		\quad & y \in Z, \quad \int_\Omega \nabla y \cdot \nabla (v - y)\dd x \geq \int_\Omega u(v - y) \dd x \quad \forall v\in Z.
	\end{aligned}
\end{equation}
Here, $\Omega \subset \R^d$, $d \in \mathbb{N}$,
is a nonempty open bounded set;
$H_0^1(\Omega)$ and $L^2(\Omega)$ are defined as usual, see
\cite{Evans2010,GilbargTrudinger1977};
$\JJ\colon H_0^1(\Omega) \times L^2(\Omega) \to \R$
is a Fr\'{e}chet differentiable objective
function
with partial derivatives 
$\partial_1 \JJ(y,u) \in H^{-1}(\Omega)$
and 
$\partial_2 \JJ(y,u) \in L^2(\Omega)$
(where $H^{-1}(\Omega)$ denotes the topological dual of
$H_0^1(\Omega)$); 
$\Uad \subset L^2(\Omega)$
is a convex, nonempty, and closed
set;
$\nabla$ is the weak 
gradient; and $Z$ is a nonempty set of the type
\[
    Z:= \left \{ v \in H_0^1(\Omega) \colon \psi_1 \leq v \leq \psi_2 \text{ a.e.\ in } \Omega\right \}
\]
involving two given measurable functions 
$\psi_1, \psi_2\colon \Omega \to [-\infty, \infty]$.
The main difficulty 
that arises when deriving 
first-order necessary optimality conditions for 
problems like 
\eqref{eq:optstaclecontrol} is that the 
governing variational inequality 
causes the 
control-to-state operator
$S\colon L^2(\Omega) \to H_0^1(\Omega)$,
$u \mapsto y$, to be nondifferentiable 
(in the sense of G\^{a}teaux and Fr\'{e}chet).
This nonsmoothness prevents classical adjoint-based approaches
as found, e.g., in \cite{Troeltzsch2010}
from being applicable
and makes it necessary to 
develop tailored strategies to establish 
stationarity systems for local minimizers.
In \cite{Mignot1976,MignotPuel1984}, the problem of deriving first-order 
optimality conditions for \eqref{eq:optstaclecontrol}
was tackled by exploiting that 
the solution mapping
$S\colon L^2(\Omega) \to H_0^1(\Omega)$, $u \mapsto y$, 
of the lower-level variational inequality  
in \eqref{eq:optstaclecontrol}
is Hadamard directionally 
differentiable with directional derivatives 
$\delta := S'(u;h)$,
$u, h \in L^2(\Omega)$, that are uniquely characterized by 
the auxiliary problem
\begin{equation}
    \label{eq:dirdiffcharobstacleproblem}
        \delta \in K_\crit(y,u), 
        \quad \int_\Omega \nabla \delta \cdot \nabla (z - \delta) 
        \dd x \geq \int_\Omega h(z - \delta) \dd x
        \quad \forall  z\in K_\crit(y,u).
    \end{equation}
Here, 
$K_\crit(y,u) := K_{\tan}(y) \cap (u + \Delta y)^\perp$
denotes the so-called \emph{critical cone}
associated with $u$ and $y := S(u)$, i.e., 
the intersection of the kernel
\[
   (u + \Delta y)^\perp  :=
   \left \{
   z \in H_0^1(\Omega)\colon
   \int_\Omega u z - \nabla y \cdot \nabla z \dd x = 0
   \right \}
\]
of the functional $u + \Delta y \in H^{-1}(\Omega)$ 
and the tangent 
cone $K_{\tan}(y) \subset H_0^1(\Omega)$ to $Z$ at $y$
which is obtained by taking the closure
of the radial cone $ K_{\rad}(y) := \R_+(Z - y)$
in $H_0^1(\Omega)$, 
cf.\ \cite[section 2]{Harder2017} and \cite{Haraux1977,Mignot1976}.
By proceeding along the lines of \cite{Mignot1976,MignotPuel1984},
one obtains the following 
main result
for
the optimal control problem \eqref{eq:optstaclecontrol}:
If a control $\bar u \in \Uad$ with state 
$\bar y := S(\bar u)$ is given 
such that the set $\R_+(\Uad - \bar u)$ is dense in $L^2(\Omega)$,
then $\bar u$ is a Bouligand stationary point of 
\eqref{eq:optstaclecontrol} in the sense that 
\begin{equation}
\label{eq:Bouligandobstacle}
\left \langle \partial_1 \JJ(\bar y, \bar u), S'(\bar u;h)\right \rangle_{H_0^1} + \left (\partial_2 \JJ(\bar y, \bar u), h \right)_{L^2}
\geq 0 \quad \forall h \in \R_+(\Uad - \bar u)
\end{equation}
holds if and only if
there exist 
an adjoint state 
$\bar p \in H^1_0(\Omega)$ and 
a multiplier 
$\bar \mu \in H^{-1}(\Omega)$ such that 
$\bar u$, $\bar y$, $\bar p$, and $\bar \mu$
satisfy the system
\begin{equation}\label{eq:sstatobst}
        \begin{gathered}
            \bar p + \partial_2 \JJ(\bar y, \bar u) = 0~~\text{ in }
            L^2(\Omega),
            \\
            - \Delta \bar p = \partial_1 \JJ(\bar y, \bar u) - \bar \mu ~~\text{ in } H^{-1}(\Omega), \\
            \bar p\in K_\crit(\bar y, \bar u), \quad 
            \left \langle \bar \mu ,z\right \rangle_{H_0^1} 
            \geq 0 \quad \forall z\in K_\crit(\bar y, \bar u).
        \end{gathered}    
\end{equation}
Here and in what follows, the symbols
$\langle \cdot, \cdot \rangle$
and $(\cdot, \cdot)$
denote a dual pairing
and a scalar product, respectively. 
For a proof of 
the above result, see
\cite[Corollary~6.1.11]{ChristofPhd2018}.
Note that, since the inequality
\eqref{eq:Bouligandobstacle}
expresses that the directional derivatives 
of the reduced objective function
$L^2(\Omega) \ni u \mapsto \JJ(S(u), u) \in \R$
of \eqref{eq:optstaclecontrol}
are nonnegative in all admissible directions 
$h \in \R_+(\Uad - \bar u)$ at $\bar u$
and thus corresponds to the most 
natural first-order 
necessary optimality condition 
obtainable for a
directionally differentiable function,
and since the conditions 
\eqref{eq:Bouligandobstacle} and \eqref{eq:sstatobst}
are equivalent, 
the system \eqref{eq:sstatobst} 
can be considered the most precise first-order 
 primal-dual necessary optimality condition
possible for \eqref{eq:optstaclecontrol}. 
This is the reason why systems of the type \eqref{eq:sstatobst}
became known as \emph{strong stationarity conditions} 
since their initial appearance in \cite{Mignot1976,MignotPuel1984}. 

The main appeal of the system \eqref{eq:sstatobst} is, of course, 
its equivalence to the
Bouligand stationarity condition  \eqref{eq:Bouligandobstacle}. This characteristic property 
distinguishes \eqref{eq:sstatobst} 
from other first-order necessary 
optimality conditions and makes \eqref{eq:sstatobst}  
an important tool, e.g., 
for assessing which information about $\bar p$ and $\bar \mu$ is lost when a
stationarity system is derived 
by means of a regularization or discretization approach.
For details on this topic, 
we refer to the survey article \cite{Harder2017}. 
Because of these advantageous properties, 
strong stationarity conditions 
have come to play a distinct role in 
the field of optimal control of nonsmooth systems and 
have received considerable attention in the recent past.
See, e.g., 
\cite{Betz2021,ChristofPhd2018,Christof2022,ReyesMeyer2016,Herzog2013,Hintermueller2009,Wachmuth2014,Wachsmuth2020} for contributions 
on strong stationarity conditions for 
optimal control problems governed by 
various 
elliptic variational inequalities of the 
first and the second kind,
\cite{Betz2019,Christof2018nonsmoothPDE,Clason2021,Meyer2017} 
for extensions to optimal control problems 
governed by nonsmooth semi- and quasilinear PDEs, and 
\cite{Christof2021} for a generalization to the multiobjective setting.
Note that all of these 
works on
the concept of strong stationarity 
have in common that they are only concerned with 
elliptic variational inequalities or 
PDEs involving nonsmooth terms.
What has --
at least to the best of our knowledge --
not been accomplished 
so far in the literature
is the derivation of
a necessary optimality condition 
analogous to \eqref{eq:sstatobst}
for an optimal control problem that is governed 
by a true evolution variational inequality (where with ``true'' we mean 
that the inequality cannot be reformulated as a nonsmooth PDE or an elliptic problem, cf.\ \cite{Betz2019}). 
In fact, such an extension is even mentioned 
as an open problem in the seminal 
works of Mignot and Puel; see \cite[section 4]{MignotPuel1984}
and \cite{MignotPuel1984:2} where 
strong stationarity conditions for parabolic obstacle 
problems are conjectured upon. 
This absence of results on strong stationarity 
systems for evolution variational inequalities 
is very unsatisfying in view of the multitude 
of processes that are modeled by this 
type of variational problem 
in finance, mechanics, and physics; 
see \cite{Mielke2015,Sofonea2012}. 

The main reason for 
the lack of contributions on strong stationarity 
conditions
for evolution variational inequalities since the 
nineteen-seventies is that directional differentiability 
results analogous to that for the elliptic obstacle 
problem in \eqref{eq:dirdiffcharobstacleproblem} have 
not been available in the
instationary setting
for a long period of time. 
See, e.g., \cite[p.\ 582]{BonnansShapiro2000} where this problem is still referred 
to as open. Only recently, progress in this direction 
has been made. In \cite{Brokate2020,Brokate2015}, it 
could be proved by means of a semi-explicit solution formula 
involving the cumulated maximum 
that the 
control-to-state operator of the problem \eqref{eq:P} 
--
the so-called \emph{scalar stop operator}
--
is Hadamard directionally differentiable in a pointwise manner;
see 
\cref{th:dirdiff} below. 
In \cite{Christof2019parob}, it could
further be shown by means 
of pointwise-a.e.\ convexity properties that 
the solution mapping of the parabolic obstacle problem 
is Hadamard directionally differentiable 
as a function into all Lebesgue spaces. 
This paper also establishes that 
the directional derivatives of the solution operator 
of the parabolic obstacle problem 
are the (not necessarily unique) solutions 
of a weakly formulated auxiliary variational inequality 
analogous to \eqref{eq:dirdiffcharobstacleproblem}, see 
\cite[Theorem 4.1]{Christof2019parob}. 
Very recently, in \cite{Brokate2021}, an auxiliary problem 
for the directional derivatives of the 
scalar stop operator in \eqref{eq:P} 
has also been obtained 
by means of a careful analysis of jump 
directions and approximation
arguments.
This auxiliary problem even yields a unique characterization, 
see \cref{th:dirdiffVI} below. 

\subsection{Main result and contribution of the paper} 
The purpose of 
the present paper is to show that the recent developments
in 
\cite{Brokate2020,Brokate2015,Brokate2021} 
make it possible to prove a
strong stationarity system for the 
optimal control problem \eqref{eq:P}.
As far as we are aware, 
our analysis is the first to establish such a system
for a true evolution variational inequality. 
The result in the literature that comes 
closest to the one derived in this paper is,
at least to the best of our knowledge, 
\cite[Theorem 5.5]{Christof2019parob}
which establishes a multiplier system 
for optimal control problems governed by 
parabolic obstacle-type variational inequalities that is 
equivalent to Bouligand stationarity 
if the adjoint state enjoys additional regularity properties
-- a deficit that is caused by a mismatch between certain 
notions of capacity, see the discussion 
in \cite[section 5]{Christof2019parob}. 
In the present work, we do not require such additional 
regularity assumptions and obtain a 
strong stationarity system for \eqref{eq:P}
that is fully equivalent to the notion of 
Bouligand stationarity.
Our main result can be summarized as follows:
If $\bar u \in \Uad$ is a control of \eqref{eq:P} 
with associated state $\bar y$ 
such that the set $\R_+(\Uad - \bar u)$ is dense in 
the control space $U$,
then $\bar u$ is a Bouligand stationary point of 
\eqref{eq:P} 
(in a sense analogous to that of
\eqref{eq:Bouligandobstacle}, see \cref{def:Bouligandstationary} below)
if and only if
there exist
an adjoint state $\bar p \in BV[0,T]$ and a multiplier $\bar \mu \in G_r[0,T]^*$
such that $\bar u$, $\bar y$, $\bar p$, and $\bar \mu$
satisfy the system
\begin{equation}
\label{eq:strongstatsys-2}
\begin{gathered}
\bar p(0) = \bar p(T) = 0,
\qquad 
\bar p(t) = \bar p(t-)~\forall t \in [0,T),
\\
\bar p(t-) \in K^\ptw_{\crit}(\bar y, \bar u)(t)~\forall t \in [0,T],
\\
\left \langle \bar \mu, z \right \rangle_{G_r} \geq 0\quad \forall z \in \KK_{G_r}^{\red,\crit}(\bar y, \bar u),
\\
\int_0^T h \dd \bar p = \left \langle \partial_3 \JJ(\bar y, \bar y(T), \bar u), h \right \rangle_{U} ~\forall h \in U,
\\
-\int_0^T z \dd \bar p  
=
\left \langle \partial_1 \JJ(\bar y, \bar y(T), \bar u), z\right \rangle_{L^\infty}
+
\partial_2 \JJ(\bar y, \bar y(T), \bar u)z(T)
-
\left \langle \bar \mu, z \right \rangle_{G_r}
\\
\hspace{8.5cm}\forall z \in G_r[0,T]. 
\end{gathered}
\end{equation}
Here, 
$BV[0,T]$ denotes the space of real-valued functions 
of bounded variation on $[0,T]$; 
$G_r[0,T]$ is the space of real-valued, regulated,
and right-continuous functions 
on $[0,T]$;
$G_r[0,T]^*$ is the topological dual space of $G_r[0,T]$;
the partial derivatives of $\JJ$
are denoted by $\partial_i \JJ$, $i=1,2,3$; 
the minus in the argument of $\bar p$ denotes a left limit;
and 
$K^\ptw_{\crit}(\bar y,\bar u)(t)$, $t \in [0,T]$, and 
$\KK_{G_r}^{\red,\crit}(\bar y, \bar u)$
are suitably defined cones (see \cref{def:ptwcritcone,def:6.1}). 
For the precise statement of the above 
result, see \cref{th:main}.
Several things are noteworthy
regarding the system \eqref{eq:strongstatsys-2}:

First of all, it can be seen that the 
adjoint state $\bar p$ lacks regularity 
in comparison with the optimal state $\bar y$ 
($BV[0,T]$ instead of $CBV[0,T]$). This reduced 
regularity reflects that the directional derivatives of the control-to-state 
mapping of \eqref{eq:P} 
are not continuous in time 
and thus significantly less regular than the 
states $y$
-- a behavior that 
is completely absent in the elliptic problem \eqref{eq:optstaclecontrol}.
For details on this topic, 
see also 
\cite[section 3]{Christof2019parob} and 
\cite[Example 4.1]{Brokate2021}
which demonstrate that all types of jump discontinuities 
of the derivatives are possible 
in the situation of \eqref{eq:P} and that
the  
derivatives cannot be expected to possess, e.g., $H^{1/2}(0,T)$-regularity,
cf.\ \cite{Jarusek2003}. 

Second, one observes that
not the adjoint state $\bar p$ 
but its left limits are
contained in the critical 
cone $K^\ptw_{\crit}(\bar y, \bar u)(t)$ for all $t \in [0,T]$ in \eqref{eq:strongstatsys-2}.
As we will see below, this condition on the limiting 
behavior 
-- 
along with the left-continuity 
of $\bar p$ in the first line of \eqref{eq:strongstatsys-2}
--
arises from certain properties of the 
jumps of the directional derivatives of the 
control-to-state mapping and the fact that the 
adjoint system evolves backwards in time
(in contrast to the variational 
inequality for the directional derivatives of the 
control-to-state mapping which evolves in a forward manner). 
Note that these additional 
properties of the left limit of the adjoint state 
are not visible in stationarity systems
derived by regularization, cf.\ 
\cite{Barbu1984,Cao2016,Colombo2020,dePinho2019,Ito2010,Stefanelli2017,Wachmuth2016}. 
This shows that \eqref{eq:strongstatsys-2} 
contains information that 
is not recoverable with regularization approaches.

Lastly, it should be noted that the coupling 
between the adjoint state $\bar p$
and the partial derivative 
$\partial_3 \JJ(\bar y, \bar y(T), \bar u)$
of the objective $\JJ$ w.r.t.\ the control in \eqref{eq:strongstatsys-2}
is not as direct as in \eqref{eq:sstatobst} but involves 
an integration step. This is a consequence of the 
rate-independence of the variational inequality 
governing \eqref{eq:P} and ultimately also the reason 
for the nonstandard start- and endpoint 
conditions $\bar p(0) = \bar p(T) = 0$ for 
$\bar p$
in \eqref{eq:strongstatsys-2}. We remark that these conditions 
reflect that the partial derivative 
$\partial_2 \JJ(\bar y, \bar y(T), \bar u)$
manifests itself -- in a distributional sense -- in the 
jump of $\bar p$ at the terminal time $T$,
see the comments at the end of \cref{sec:7}. 
A similar behavior can also be observed 
in optimal control problems for parabolic PDEs, 
see \cite[section 5.5.1]{Troeltzsch2010}. 

Regarding the derivation of 
the strong stationarity system 
in \cref{th:main}, 
we would like to point out that --
even with the results of 
\cite{Brokate2020,Brokate2015,Brokate2021} 
at hand and 
even though the variational inequality in \eqref{eq:P}
is one of the simplest evolution 
variational inequalities imaginable -- the proof
of \eqref{eq:strongstatsys-2} is still quite involved. The main difficulty 
in the context of \eqref{eq:P} is that, 
due to the lack of weak-star continuity properties of 
the scalar stop operator, 
one has to discuss this problem in a 
control space $U$ whose topology is significantly 
stronger than that of $BV[0,T]$
to be able to ensure that \eqref{eq:P} 
is well posed;
see the comments in \cref{sec:5} below.
Since the directional derivatives 
of the scalar stop are only in $BV[0,T]$,
the need for such a ``small'' control space $U$ 
makes it necessary 
to employ a careful limit analysis to ensure 
that the control space is 
\emph{ample} enough to be able to arrive
at a strong stationarity system.
Compare also with the comments 
on this topic in \cite{Christof2022,Herzog2013} and 
the results in \cite{Wachmuth2014} in this context. 
In our analysis, 
we tackle this problem by 
generalizing the classical 
concept of \emph{polyhedricity}
to the time-dependent setting.
This is a density property which,
in the situation of the 
elliptic problem \eqref{eq:optstaclecontrol}, 
ensures that 
the set of critical radial 
directions $ K_{\rad}(y) \cap (u + \Delta y)^\perp$
is $H_0^1(\Omega)$-dense in 
$K_{\tan}(y) \cap (u + \Delta y)^\perp$
and which plays an important role 
in the sensitivity analysis 
of elliptic obstacle-type variational inequalities
as well as the theory of second-order optimality conditions, 
see \cite{Haraux1977,Christof2018SSC,Wachsmuth2019}. For the approximation result 
that we establish in this context and that 
we refer to as ``temporal polyhedricity'',
see \cref{theorem:tempoly}. 

We expect that \cref{theorem:tempoly},
along with the insights provided by \eqref{eq:strongstatsys-2},
is also helpful for the analysis of optimal control problems 
governed by more complicated evolution variational inequalities, 
cf.\ the problems
studied in \cite{Christof2019parob,Muench2018,Samsonyuk2019}.

\subsection{Structure
of the remainder of the paper}
We conclude this section with an
overview of the 
content and the structure 
of the remainder of the paper.

\Cref{sec:2,sec:3} are concerned with preliminaries. 
Here, we introduce the notation and the standing assumptions that 
we use throughout this work. In \cref{sec:4}, we collect 
basic results on the properties of the control-to-state mapping of \eqref{eq:P}
-- the scalar stop operator. This section also
recalls the directional differentiability results of \cite{Brokate2020,Brokate2015,Brokate2021} and 
discusses some of their consequences. \Cref{sec:5} 
addresses the solvability of \eqref{eq:P}
and introduces the concept of Bouligand stationarity 
for this problem. This section also contains an example  which shows 
that,
to be able to prove the existence of solutions for \eqref{eq:P}
by means of the direct method of the calculus of variations, 
one indeed has to consider a control space significantly smaller than 
$BV[0,T]$. In \cref{sec:6}, we prove the already mentioned 
temporal polyhedricity property for \eqref{eq:P}.
The main result of this section is \cref{theorem:tempoly}.
\Cref{sec:7} is concerned with 
the proof of the strong stationarity system 
\eqref{eq:strongstatsys-2}, see \cref{th:main}.
The \hyperref[sec:appendix]{appendix}
of the paper collects 
some results on the Kurzweil-Stieltjes integral 
that are needed for our analysis.

\section{Notation}
\label{sec:2}
Throughout this work, $T>0$ is a given and fixed number.  
We denote the space of 
real-valued continuous functions on $[0,T]$
by $C[0,T]$ and the space of real-valued 
regulated functions on $[0,T]$ (i.e., the space of all functions that are uniform limits of 
step functions, see \cite[Definition 4.1.1, Theorem 4.1.5]{Monteiro2019}) by $G[0,T]$.
We equip both $C[0,T]$ and $G[0,T]$
with the supremum norm $\|\cdot\|_\infty$.
Recall that this makes  $C[0,T]$ and $G[0,T]$ Banach spaces 
and that every $v \in G[0,T]$
possesses left and right limits,
see \cite[chapter 4]{Monteiro2019}.
Given $v \in G[0,T]$, we denote these limits 
by $v(t-)$ and $v(t+)$, respectively, with the 
usual conventions at the endpoints of $[0,T]$, i.e., 
\begin{align*}
v(t-) &:= \lim_{[0, T] \ni s \to t^-} v(s)\quad \forall t \in (0, T],\qquad v(0-) := v(0),
\\
v(t+) &:= \lim_{[0, T] \ni s \to t^+} v(s)\quad \forall t \in [0, T),\qquad v(T+) := v(T).
\end{align*}
For the left- and the right-limit function
associated with a function $v \in G[0,T]$, 
we use the symbols $v_-$ and $v_+$, 
i.e., $v_-(t) := v(t-)$ and $v_+(t) := v(t+)$
for all $t \in [0,T]$. We further define 
$G_r[0,T] := \left \{ v \in G[0,T] \colon v = v_+ \right \}$.
It is easy to check that this set of 
right-continuous regulated functions is a closed subspace of 
$(G[0,T], \|\cdot\|_\infty)$. 

The space of real-valued functions of bounded variation on $[0,T]$ is denoted by 
$BV[0,T]$. We emphasize that we do not 
consider elements of $BV[0,T]$ as equivalence classes in this paper but as classical functions $v\colon [0,T] \to \R$,
as in \cite[chapter~2]{Monteiro2019}.
For a discussion of different 
approaches to $BV[0,T]$, 
see \cite{Ambrosio2000}.
We denote the variation of a function $v\colon[0,T]\to\R$ 
by $\var(v)$,
and we define
the total variation norm on $BV[0,T]$ as 
$\|v\|_{BV} := |v(0)| + \var(v)$.
Recall that $(BV[0,T], \|\cdot \|_{BV})$ is 
a Banach space that 
is continuously embedded into $(G[0,T], \|\cdot \|_{\infty})$;
see \cite[Theorem~2.2.2]{Monteiro2019}.
We define 
$CBV[0,T] := BV[0,T] \cap C[0,T]$ and 
$BV_r[0,T] := BV[0,T] \cap G_r[0,T]$.
Note that both of these sets are 
closed subspaces of $(BV[0,T], \|\cdot \|_{BV})$.

Given a set-valued function 
$K\colon [0,T] \rightrightarrows \R$
and $0 \leq s < \tau \leq T$, 
we use the symbols $C([s,\tau];K)$ and $G([s,\tau];K)$
to denote the sets
of continuous and regulated functions $v$ on $[s,\tau]$
which satisfy $v(t) \in K(t)$ for all $t \in [s,\tau]$,
respectively. Sets $K \subset \R$ are interpreted 
as set-valued functions that are constant in time in 
this notation. We further set
$C^\infty[0,T] := \{v \in C[0,T] \colon 
\exists \tilde v \in C^\infty(\R)\, \text{s.t.}\,
v(t) = \tilde v(t) ~\forall t \in [0,T]\}$. 
For the classical Lebesgue and Sobolev 
spaces, 
we use the standard notation 
\mbox{$(L^p(0,T), \|\cdot\|_{L^p})$}
and 
$(W^{k,p}(0,T), \|\cdot\|_{W^{k,p}})$,
$1 \leq p \leq \infty$, $k \in \mathbb{N}$. 
The weak derivative of a function $v \in W^{1,p}(0,T)$
is denoted by $v' \in L^p(0,T)$.
For the topological dual of a 
normed space $(X, \|\cdot\|_X)$, we use the symbol $X^*$,
and for a dual pairing, the brackets $\langle\cdot, \cdot \rangle$
equipped with a subscript that clarifies 
the space.
A closure is denoted by $\closure(\cdot)$.
Weak, weak-star,
and strong convergence
are indicated by
$\weakly$, \smash{$\weaklystar$}, and $\to$, respectively. 
Given a set $D \subset [0,T]$, 
we define $\mathds{1}_{D}\colon [0,T] \to \{0,1\}$
to be the characteristic function of $D$, i.e., the 
function that equals 1 on $D$ and 0 everywhere else. 

\section{Main problem and standing assumptions}
\label{sec:3}
As already mentioned in the introduction, 
the aim of this paper is to study 
optimal control problems of the type
\begin{equation*}
\tag{P}
\left \{~~
	\begin{aligned}
		\text{Minimize} \quad &  \JJ(y, y(T), u)  \\
		\text{w.r.t.}\quad &y \in CBV[0, T], \quad  u \in \Uad,\\
		\text{s.t.}\quad & y = \S(u),
	\end{aligned}
\right.
\end{equation*}
where $\S$ is the
\emph{scalar stop operator}, i.e., 
the solution map
$\S\colon CBV[0,T] \to CBV[0,T]$, $u \mapsto y$, of the rate-independent 
evolution variational inequality 
\begin{equation}
\label{eq:V}
\tag{V}
\left \{~~
\begin{aligned}
&\int_0^T (v - y)\dd (y - u)\geq 0 &&\forall v \in C([0, T]; Z),
\\
&y(t) \in Z\quad \forall t \in [0, T], &&y(0) = y_0.
\end{aligned}
\right.
\end{equation}
General references for the properties of
the function $\S$ are \cite{BrokateSprekels1996,Krejci1996,Krejci1999}; some of them
will be discussed in detail in \cref{sec:4}. 

Note that, from the application point of view, 
\eqref{eq:P} can be interpreted as an optimal control problem 
for a one-dimensional sweeping process with characteristic set $Z = [-r,r]$,
i.e., a problem that aims to control the trajectory of a 
body with one degree of freedom that is placed on a slippery surface 
within $Z$ and moved (swept) by moving $Z$ back and forth, see \cite[section 1.1]{Mielke2015}. 
(In this case, the trajectory is described by the \emph{scalar play operator} $\PP(u) := u - \S(u)$
and the control function $u$ models the movement of $Z$.)
This physical interpretation, however, is mainly secondary in this work. 
We are primarily interested in the problem \eqref{eq:P} because it is the 
instationary counterpart of the optimal control problem \eqref{eq:optstaclecontrol} 
for the elliptic obstacle problem and captures the effects of ``pure'' evolution without 
any additional spatial dependencies (as present, e.g., in the parabolic obstacle problem, cf.\
\cite{Christof2018SSC}). We hope that the insights provided by our analysis are also helpful for the analysis 
of optimal control problems governed by more complicated systems arising, e.g., in the field of 
elasto-plasticity, which often involve the play and stop operator to incorporate hysteresis effects, 
cf.\ \cite{Mielke2015, Muench2018,Samsonyuk2019}.

We would like to emphasize that the integral in \eqref{eq:V} 
-- along with all other integrals appearing in 
the remainder of this paper
--
is to be understood in the sense 
of Kurzweil-Stieltjes.
For an in-depth introduction to the integration theory 
for this type of integral,
we refer to \cite{Monteiro2019}.
A collection of basic definitions, elementary properties, 
and fundamental results related to the Kurzweil-Stieltjes integral can also be found in the  \hyperref[sec:appendix]{appendix} of this paper. 
The use of the 
Kurzweil-Stieltjes integral
for the variational inequality approach to rate-independent evolutions goes back to \cite{Krejci2003,Krejci2006,KrejciLiero2009} where it was employed for the study of discontinuous
input functions $u$. 
For this kind of $u$, the integrand and the integrator (i.e., the function behind the ``$\dd$'')
in \eqref{eq:V}
usually have discontinuities at common points $t\in [0,T]$
so that the Riemann-Stieltjes integral no longer works. 
Such common discontinuities also appear naturally in the variational inequality that characterizes the directional derivatives of $\S$, cf.\ \cref{th:dirdiffVI} below.
For a treatment based on the Young integral, see \cite{KrejciLaurencot2002}. 
Alternatively, the Lebesgue-Stieltjes integral can be used as for the types of integrands and integrators appearing 
in this paper it is equivalent to the Kurzweil-Stieltjes integral, see \cite[section 6.12]{Monteiro2019}. 

However,
for this type of integral, a careful handling of statements involving ``almost everywhere'' is necessary
since the $\sigma$-algebra and the family of its sets of measure zero depend on the integrator. In particular, a singleton $\{t\}$ has nonzero measure if the integrator is discontinuous at $t$.

For the ease of reference, we collect our standing assumptions 
on the quantities in the optimal 
control problem \eqref{eq:P} and 
the variational inequality \eqref{eq:V} in:

\begin{assumption}[standing assumptions]~\label{ass:standing}
\begin{itemize}
\item $T>0$ is given and fixed.
\item $U \subset CBV[0,T]$ is a 
real vector space that is endowed with a norm $\|\cdot\|_U$
and that is continuously and densely embedded into $(C[0,T],\|\cdot\|_\infty)$.
\item $\Uad$ is a nonempty and convex subset of $U$.
\item $\JJ\colon L^\infty(0, T) \times \R \times U \to \R$ 
is a Fr\'{e}chet differentiable function 
whose partial derivative w.r.t.\ 
the first argument satisfies
$\partial_1 \JJ(y,  y(T), u) \in L^1(0, T)
$ for all $(y, u) \in CBV[0,T] \times U$. 
Here, $L^1(0,T)$ is interpreted as a subset of 
$L^\infty(0,T)^*$ via the canonical embedding 
into the bidual. 
\item 
$Z$ is an 
interval of the form $Z = [-r,r]$ with an 
arbitrary but fixed $r>0$.
\item $y_0 \in Z$ is a given and fixed starting value.
\end{itemize}
\end{assumption}

The above assumptions are 
always assumed to hold 
in the following sections, even when not explicitly mentioned. 
We remark that, to be able to prove the existence of 
solutions for \eqref{eq:P}, one requires more 
information about $\JJ$, $\Uad$, etc.\ than 
provided by \cref{ass:standing};
see \cref{cor:solex}. For the derivation of the 
strong stationarity system \eqref{eq:strongstatsys-2}, 
however, 
this is not relevant.
An example of a control space $U$ that satisfies the conditions 
in \cref{ass:standing} and that allows to prove the existence of 
minimizers for \eqref{eq:P} is the space $H^1(0,T)$, see \cref{sec:5} and the comments therein.

\section{Properties of the scalar stop operator 
\texorpdfstring{$\boldsymbol{\S}$}{S}}%
\label{sec:4}%
In this section, we collect properties of the solution map
$\S\colon u \mapsto y$ of the variational inequality 
\eqref{eq:V} that are needed 
for our analysis. We begin with 
fundamental results on the well-definedness,
monotonicity, and directional differentiability of $\S$. 

\begin{theorem}[well-definedness and Lipschitz continuity]%
\label{th:Swellposed}%
The variational inequality \eqref{eq:V} possesses a unique solution 
$\S(u) := y\in CBV[0, T]$ for all $u \in CBV[0, T]$.
For all $u\in W^{1,1}(0, T)$, 
it holds $y = \S(u) \in W^{1,1}(0, T)$ and
\begin{equation}\label{eq;Vabscont}
(v - y(t)) (y'(t) - u'(t)) \ge 0 \qquad \forall v\in Z  
\qquad\text{for a.a.}~t \in (0,T). 
\end{equation}
Further, $\S$ satisfies the Lipschitz estimate
\begin{equation}
\label{eq;LipschitzSinfty}
\|\S(u_1) - \S(u_2)\|_\infty \leq 2\|u_1 - u_2\|_\infty\qquad \forall u_1, u_2 \in CBV[0, T].
\end{equation}
\end{theorem}

\begin{proof}
Proofs of the unique solvability of \eqref{eq:V}
in $CBV[0, T]$ and of \eqref{eq;Vabscont} can be found in
\cite[Theorem 4.1, Proposition 4.1]{Krejci1999}.
The Lipschitz estimate \eqref{eq;LipschitzSinfty} 
follows from \cite[Theorem 7.1]{Krejci1999};
see also \cite[p.\ 49f.]{Krejci1996} and 
\cite[Proposition 2.3.4]{BrokateSprekels1996}.
\end{proof}

\begin{lemma}[general test functions]\label{lemma:gentest}%
Let $u\in CBV[0,T]$ and $0\le s < \tau\le T$. Then $y := \S(u)$
satisfies
\begin{equation}\label{eq:gentest}
\int_s^\tau (v-y)\dd(y-u) \ge 0 \qquad \forall v\in G([s,\tau];Z).   
\end{equation}
\end{lemma}

\begin{proof}
Since  $y + \mathds{1}_{[s,\tau]}(v-y) \in G([0,T];Z)$
for all $v \in G([s,\tau];Z)$ and due to \cref{lemma:subintervals},
it suffices to consider the case $[s,\tau] = [0,T]$.
Let $v\colon[0,T]\to Z$ be a step function of the form 
\[
v = \sum_{j=1}^N \mathds{1}_{(t_{j-1},t_j)}\zeta_j + \sum_{j=0}^N \mathds{1}_{\{t_j\}}\hat{\zeta}_j
\]
with $\zeta_j,\hat{\zeta}_j\in Z$ and $0 = t_0 < ... < t_N = T$.  
Since $v = \lim_{n\to \infty} v_n$ pointwise for suitable $v_n\in C([0,T];Z)$, \eqref{eq:gentest} 
for $v$
follows from the bounded convergence theorem, \cref{th:boundedconv}.
As step functions are dense in $G([0,T];Z)$
by \cite[Theorem 4.1.5]{Monteiro2019},
\eqref{eq:gentest} holds for arbitrary $v\in G([0,T];Z)$, again by the bounded convergence theorem.
\end{proof}

\begin{lemma}[piecewise monotonicity]\label{lemma:yumon}%
Let $u\in CBV[0,T]$
and set $y := \S(u)$.
Let $J$ be an open nonempty subinterval of $[0,T]$.
\begin{enumerate}[label=\roman*)]
\item\label{lemma:yumon:item:i}
If $J \subset \{ t\in [0,T]\colon y(t) > -r\}$, then $y-u$ is nonincreasing on $\closure{(J)}$. 
\item\label{lemma:yumon:item:ii}
If $J\subset \{ t\in [0,T] \colon y(t) < r\}$, then 
$y-u$ is nondecreasing on $\closure{(J)}$.
\end{enumerate}
\end{lemma}

\begin{proof}
We prove \ref{lemma:yumon:item:i}. (The proof of \ref{lemma:yumon:item:ii} 
is analogous.) Let $s,\tau\in J$ with $s < \tau$.
Then $y \ge - r + \varepsilon$ on $[s,\tau]$ for some $\varepsilon > 0$. 
As $v: = y - \varepsilon \in G([s,\tau];Z)$, we can apply \cref{lemma:gentest} to obtain
\[
0 \le \int_s^\tau (v-y)\dd(y-u) = - \varepsilon ((y-u)(\tau) - (y-u)(s)).
\]
Thus, $y-u$ is nonincreasing on $J$, and hence on $\closure{(J)}$ since $y-u$ is continuous.
\end{proof}

A proof of the foregoing lemma based on an explicit representation of $y-u$ can be found in \cite[section 5]{Brokate2015}.

\begin{lemma}[comparison principle]%
\label{lemma:monotonicity}%
Let $u_1, u_2 \in CBV[0, T]$
be given such that $u_2 - u_1$ 
is nondecreasing in $[0, T]$.
Then it holds $\S(u_2)(t) \geq \S(u_1)(t)$ for all $t \in [0, T]$.
\end{lemma}

\begin{proof}
First, let us assume that $u_1,u_2\in W^{1,1}(0,T)$.
From \eqref{eq;Vabscont}, we obtain that $y_1 := \S(u_1)$ and $y_2 := \S(u_2)$ satisfy
\begin{equation}
\label{eq:randomeq2736}
(v  - y_i(t))(y_i'(t)- u_i'(t))\geq 0\qquad \forall v \in Z \qquad 
\text{for a.a.}~t \in (0, T) \qquad i=1,2.
\end{equation}
Testing \eqref{eq:randomeq2736} for $i=1$ with 
$v = y_1(t) - \max\{0, y_1(t) - y_2(t)\} \in Z$
and for $i=2$ with $v = y_2(t) + \max\{0, y_1(t) - y_2(t)\} \in Z$
and adding the resulting inequalities gives
\[
\max\{0, y_1  - y_2 \}\cdot (y_1' - y_2') 
 \leq \max\{0, y_1 - y_2\}\cdot (u_1' - u_2') \leq 0
 \quad \text{a.e.\ in $(0,T)$}
\]
as $u_2 - u_1$ is nondecreasing.
By a classical result of Stampacchia, see, for instance,
\cite[Lemmas~7.5 and 7.6]{GilbargTrudinger1977},
we have
\[
\ddt \frac12 \big(
 \max\{0, y_1 - y_2\}\big)^2 = 
 \max\{0, y_1 - y_2\}\cdot(y_1' - y_2')
 \qquad \text{a.e.\ in $(0,T)$.}
\]
Since $y_2(0) = y_1(0)$, 
we conclude that $\max\{0, y_1 - y_2\} \le 0$ on $[0,T]$. Thus, $y_2 \ge y_1$ on $[0,T]$ as claimed.
In the general case $u_1,u_2\in CBV[0, T]$,
we choose piecewise affine interpolants $u_1^n,u_2^n$ of $u_1,u_2$ on
partitions $\Delta_n$ of $[0,T]$ whose widths go to zero for $n \to \infty$. Since $u_2^n - u_1^n$ is nondecreasing, too,
it follows that $\S(u_2^n) \ge \S(u_1^n)$ on $[0,T]$ for all $n$. As $u_i^n \to u_i$ uniformly, by virtue 
of \eqref{eq;LipschitzSinfty}, we may pass to the limit, 
and the claim follows.
\end{proof}

\begin{theorem}[pointwise directional differentiability of $\S$]%
\label{th:dirdiff}%
The solution operator $\S\colon CBV[0, T]  \to CBV[0, T]$
of \eqref{eq:V} is pointwise directionally differentiable
in the sense that, 
for all $u, h \in CBV[0, T]$, there is a unique  $\S'(u;h) \in BV[0, T]$
satisfying 
\[
\lim_{\alpha \to 0^+} \frac{\S(u + \alpha h)(t) - \S(u)(t)}{\alpha} = \S'(u;h)(t)\qquad \forall t \in [0,T].
\]
\end{theorem}

\begin{proof}
See \cite[Corollary 5.4, Proposition 6.3]{Brokate2015} and also \cite[Theorem 2.1]{Brokate2021}.
\end{proof}\pagebreak

Similarly to the classical result \eqref{eq:dirdiffcharobstacleproblem}
for the obstacle problem,
the derivatives $\S'(u;h)$
in \cref{th:dirdiff} are characterized by an auxiliary 
variational inequality. To be able to state this 
inequality, we require some additional notation
from \cite{Brokate2021}.

\begin{definition}[inactive, biactive, and 
strictly active set]\label{def:biactiveetc}%
Let $u \in CBV[0,T]$ be a control with 
state $y := \S(u) \in CBV[0,T]$. 
We introduce:
\begin{itemize}
\item the inactive set:
\[
I(y) := \{t \in [0, T] \colon |y (t)| < r\},
\]
\item the biactive set associated with the upper bound of $Z$:
\[
B_+(y,u):= \{t \in [0, T] \colon y (t) = r \text{ and } \exists \varepsilon > 0 \text{ s.t. } 
 y -  u = \mathrm{const}\text{ on } [t, t + \varepsilon)\},
\]
\item the biactive set associated with the lower bound of $Z$:
\[
B_-(y,u) := \{t \in [0, T] \colon  y (t) = -r \text{ and } \exists \varepsilon > 0 \text{ s.t. } 
 y - u = \mathrm{const}\text{ on } [t, t + \varepsilon)\},
\]
\item the biactive set:
\[
B(y,u) := B_+(y,u) \cup B_-(y,u),
\]
\item the strictly active set:
\[
 A(y,u) := \{t \in [0, T) \colon | y (t)| = r \text{ and } \nexists \varepsilon > 0 \text{ s.t. } 
 y - u = \mathrm{const}\text{ on } [t, t + \varepsilon)\}.
\]
\end{itemize}
Here and in what follows, 
we use the convention $T \in B_\pm(y,u)$ in the case $y(T) = \pm r$.
\end{definition} 

\begin{definition}[radial and critical cone mapping]
\label{def:ptwcritcone}%
Given an input function $u \in CBV[0,T]$ with 
state $y := \S(u) \in CBV[0,T]$, we define:
\begin{itemize}
\item the set-valued pointwise radial cone mapping:
\[
K_\rad^\ptw(y)\colon [0, T] \rightrightarrows \R,
\qquad
K_\rad^\ptw(y)(t):=
\begin{cases}
\R & \text{ if } |y(t)| < r,
\\
(-\infty, 0] & \text{ if } y(t) = r,
\\
 [0, \infty) & \text{ if } y(t) = -r,
 \end{cases}
\]
\item the set-valued pointwise critical cone mapping:
\[
K^\ptw_{\crit}(y,u) \colon [0, T] \rightrightarrows \R,
\qquad
K^\ptw_{\crit}(y,u) (t):=
\begin{cases}
\R & \text{ if } t \in I(y),
\\
(-\infty, 0] & \text{ if } t \in B_+(y,u),
\\
 [0, \infty) & \text{ if } t \in B_-(y,u),
 \\
 \{0\} & \text{ if } t \in  A(y,u).
 \end{cases}
\]
\end{itemize}
\end{definition}

Obviously,
\[
K^\ptw_{\crit}(y,u)(t) 
\subset K_\rad^\ptw(y)(t) \quad\forall 
t \in [0,T].
\]

Note that 
a function $z \in C^\infty[0, T]$ satisfying $z(t) \in K_\rad^\ptw(y)(t)$ for all $t \in [0, T]$
is not necessarily an element of the ``global'' 
radial cone associated with \eqref{eq:V}, i.e., 
does not necessarily satisfy $y(t) + \alpha z(t) \in Z$ for all $t \in [0, T]$ for a number $\alpha > 0$
independent of $t$. 
A possible counterexample here is $r=1$, $y_0 = 0$, $T = \pi/2$, $y(t) = u(t) = \sin(t)$,
and $z(t) = \sin(2t)$. 
Indeed, for these $r$, $y_0$, $T$, $y$, $u$, and $z$, we clearly have $y = \S(u)$,
$z(T) = 0 \in K_\rad^\ptw(y)(T) = (-\infty, 0]$, and 
$z(t) \in K_\rad^\ptw(y)(t) = \R$ for all $t \in [0, T)$.
Due to the identities $y(T) = 1 = r$, 
$y'(T) = 0$, and $z'(T) = -2$, it further holds $y'(T) + \alpha z'(T) = - 2 \alpha < 0$ for all $\alpha > 0$.
This implies that, for all $\alpha > 0$, there exists $t \in [0,T]$ satisfying $y(t) + \alpha z(t) > r$.
We thus have $z(t) \in K_\rad^\ptw(y)(t)$ for all $t \in [0,T]$ but there does not exist 
$\alpha > 0$ satisfying $y(t) + \alpha z(t) \in Z$ for all $t \in [0, T]$.

\cref{prop:classicalcritical,cor:classicalcritical} 
below establish a connection between the 
pointwise critical cone mapping 
$K^\ptw_{\crit}(y,u)\colon [0,T] \rightrightarrows \R$
and the classical notion of criticality, 
that is, 
the property of being an element of the kernel of 
the multiplier that appears in the 
variational inequality \eqref{eq:V}, cf.\ 
the definition of $K_\crit(y,u)$ in \eqref{eq:dirdiffcharobstacleproblem}.
As a preparation for \cref{prop:classicalcritical,cor:classicalcritical}, 
we prove the following lemma.

\begin{lemma}\label{lemma:ortho}
Let $u \in CBV[0,T]$ be a control with state $y := \S(u) \in CBV[0,T]$ and let $z\in G[0,T]$ 
be a function satisfying $z = 0$ on $ A(y,u)$. Then
\begin{equation}\label{eq:ortho}
\int_s^\tau z\dd (y-u) = 0 \qquad \forall\;0\le s < \tau\le T.
\end{equation}
\end{lemma}

\begin{proof}
Define 
$D := \left ( I(y) \cup B(y,u) \right )\setminus \{T\}$. 
The continuity of $y$ and the definitions
of $I(y)$ and $B(y,u)$ imply that, for every $t \in D$, there exists  
$\varepsilon > 0$ with $[t, t + \varepsilon) \subset D$. This entails that the set $D$
decomposes into disjoint connected components 
$\{D_i\}_{i \in \II}$
with $\II$ being finite or equal to $\mathbb{N}$
and $D_i$ being an interval with a 
nonempty interior for all $i \in \II$.
Using \cref{lemma:yumon} and again the definition of $B(y,u)$, one easily checks that,
for each $t \in D$, there exists $\varepsilon > 0$ such that 
$y - u$ is constant on $[t, t + \varepsilon)$.
Since $y-u$ is continuous, 
this implies $y-u =: c_i = \mathrm{const}$ 
on each $[a_i,b_i] := \closure{(D_i)}$.
Now $\smash{\int_0^T \mathds{1}_{\{T\}}z\dd(y-u) = 0}$ by \cref{eq:singlepointmass}.
Using this identity, 
the fact that $z=0$ holds on $ A(y,u)$,
\cref{lemma:subintervals}, and (in the case 
$\II = \mathbb{N}$) the bounded convergence theorem (\cref{th:boundedconv}), we see that
\begin{equation}
    \label{eq:randomeq2536}
\begin{aligned}
\int_0^T z\dd (y-u)
&= \int_0^T \mathds{1}_D z\dd (y-u)
= \int_0^T \sum_{i\in \II} \mathds{1}_{D_i} z\dd (y-u)
\\ &= \sum_{i\in \II} \int_0^T \mathds{1}_{D_i} z\dd (y-u)
= \sum_{i\in \II} \int_{a_i}^{b_i} \mathds{1}_{D_i} z\dd c_i
= 0.
\end{aligned}
\end{equation}
Choosing 
$\mathds{1}_{[s,\tau]}z$ instead of $z$
in \eqref{eq:randomeq2536}
yields \eqref{eq:ortho}, again due to \cref{lemma:subintervals}.
\end{proof}

\begin{proposition}[relation to the classical notion of criticality]
\label{prop:classicalcritical}%
Suppose that a control 
$u \in CBV[0,T]$ with state $y := \S(u) \in CBV[0,T]$
and a function $z\in G[0,T]$ 
satisfying $z(t) \in K_\rad^\ptw(y)(t)$
for all $t \in [0,T]$ are given.
Then it holds
\begin{equation}\label{eq:classicalcritical:item:ii}
\int_s^\tau z\dd (y-u) \ge 0 \qquad \forall\;0\le s < \tau\le T.
\end{equation}
Moreover, it is true that \vspace{0.05cm}
\begin{equation}\label{eq:classicalcritical:item:i}
z(t) \in K^\ptw_{\crit}(y,u)(t)~\forall t \in [0,T]
\quad\Rightarrow\quad
\int_s^\tau z\dd(y-u) = 0~~\forall\,0 \leq s < \tau \leq T,
\end{equation}
and, if $z$ possesses the additional regularity
$z\in G_r[0,T]$,
then we also have 
\begin{equation}\label{eq:classicalcritical:item:iii}
\int_0^T z\dd(y-u) = 0
\quad \Rightarrow\quad
z(t) \in K^\ptw_{\crit}(y,u)(t)~\forall t \in [0,T].
\end{equation}
\end{proposition}

\begin{proof}
In order to prove \eqref{eq:classicalcritical:item:ii},
let $0 \leq s < \tau \leq T$ be given. 
We first assume that \mbox{$y(t) \in (-r,r]$} holds 
for all $t \in [s, \tau]$. 
By \cref{lemma:yumon}, $u - y$ is nondecreasing on $[s, \tau]$. 
Using the definition of $K_\rad^\ptw(y)$, it is easy to check that 
$\hat{z}(t) := \max\{0,z(t)\} \mathds{1}_{[s,\tau]}(t)$ satisfies
the assumptions of \cref{lemma:ortho}. Therefore,
\[
\int_s^\tau z\dd(y - u) 
= 
\int_s^\tau  \min\{0,z\} \dd(y - u)
=
\int_s^\tau \max\{0,-z\} \dd(u - y) \geq 0. 
\]
This proves \eqref{eq:classicalcritical:item:ii}
in the case $y(t) \in (-r,r]$ for all $t \in [s, \tau]$.
In the case $y(t) \in [-r,r)$  for all $t \in [s, \tau]$,
we can use the exact same arguments as above with reversed
signs to establish \eqref{eq:classicalcritical:item:ii}. 
To finally obtain \eqref{eq:classicalcritical:item:ii}
for arbitrary $[s,\tau]$, it suffices to consider 
a subdivision of $[s,\tau]$ into subintervals of the above two 
types and to use \eqref{eq:decomposeInterval}. 

The implication \eqref{eq:classicalcritical:item:i} 
follows directly from \cref{lemma:ortho} since 
$z(t)\in K^\ptw_{\crit}(y,u)(t)$ for all $t \in [0,T]$
implies $z = 0$ on $ A(y,u)$.

It remains to prove \eqref{eq:classicalcritical:item:iii}.
Since $z(t)\in K_\rad^\ptw(y)(t) \setminus K^\ptw_{\crit}(y,u)(t)$ for some $t \in [0,T]$
if and only if $z(t) \neq 0$
and $t\in  A(y,u)$, it
suffices to show that the integral on the left side of
\eqref{eq:classicalcritical:item:iii} is nonzero 
if a time $t$ with the latter property exists. 
So let $t \in  A(y,u)$ be arbitrary but fixed and suppose that 
$z(t) \neq 0$. We assume w.l.o.g.\ that $y(t) = r$. (The case $y(t) = -r$ is analogous.)
From $0 \neq z(t) \in K_\rad^\ptw(y)(t)$, we obtain that $z(t) < 0$ holds,
and from the right-continuity of $z$, the definition of $ A(y,u)$, and the continuity of $y$,
that $t \neq T$ and that there exist numbers $c, \varepsilon > 0$ such that $z(s) \leq -c $ and $y(s) \in (-r, r]$ 
holds for all
$s \in [t, t + \varepsilon] \subset [0, T]$ and 
such that $y-u$ is not constant on $[t,t+\varepsilon)$. 
By \cref{lemma:yumon}, $y-u$ is nonincreasing on $[t, t + \varepsilon]$. It thus follows that
\[
\int_{t}^{t + \varepsilon} z \dd(y-u) \geq c \int_{t}^{t+\varepsilon} \dd(u-y)
= c\left ( (u - y)(t + \varepsilon)  - (u - y)(t)\right ) > 0.
\]
Using \eqref{eq:classicalcritical:item:ii}, we conclude
\[
\int_{0}^{T} z \dd(y-u)  
= \int_{0}^{t} z \dd(y-u) + \int_{t}^{t + \varepsilon} z \dd(y-u)  + \int_{t + \varepsilon}^{T} z \dd(y-u) 
> 0.
\]
\end{proof}

\begin{corollary}\label{cor:classicalcritical}
Let $u \in CBV[0,T]$ be a control with state $y := \S(u)$
and let $z\in G_r[0,T]$ be a given function. 
Then
\begin{equation*}
 z(t) \in K^\ptw_{\crit}(y,u)(t)~\forall t \in [0,T]\quad\Leftrightarrow\quad
 \left \{~~
 \begin{aligned}
&z(t) \in K^\ptw_{\rad}(y)(t)~\forall t \in [0,T] \text{ and } 
\\
&
\int_s^\tau z\dd(y-u) = 0~~\forall\,0 \leq s < \tau \leq T. \end{aligned}\right.
\end{equation*}
\end{corollary}

As \cref{cor:classicalcritical} shows, 
a function $z \in G_r[0,T]$ is 
``critical in the pointwise sense'' 
if and only if it takes values in
$\smash{K_\rad^\ptw(y)(t)}$ for all $t\in [0,T]$ 
and is contained in the kernel of the 
linear and continuous function
$
\smash{G[0,T] \ni v \mapsto \int_s^\tau v \dd(y-u) \in \R}
$
for all $0 \leq s < \tau \leq T$.
For elements of $G_r[0,T]$, 
the pointwise notion of criticality introduced in 
\cref{def:ptwcritcone} is thus 
closely related to the 
notion of criticality appearing in the context of the 
classical obstacle problem, cf.\ \eqref{eq:dirdiffcharobstacleproblem}. 
This relation does not exist anymore in general
when the assumption of right-continuity is dropped. 
Indeed, as the integrator $y-u$
of the integrals in \cref{prop:classicalcritical,cor:classicalcritical} does not assign mass to singletons due to the continuity
of $u$ and $y$
and \cref{eq:singlepointmass}, 
for every $t \in  A(y,u)$, the function 
$z(s) := -\sgn(y(t)) \mathds{1}_{\{t\}}(s)$
satisfies $z \in G[0, T]$, 
\smash{$z(s) \in K_\rad^\ptw(y)(s)$} for all $s \in [0, T]$,
and $\smash{\int_s^\tau z\dd(y - u) = 0}$
for all $0 \leq s < \tau \leq T$ but does not vanish 
on the strictly active set $ A(y,u)$. In all situations in which $ A(y,u)$ is nonempty,
the pointwise notion of criticality 
in \cref{def:ptwcritcone} thus differs from the 
ordinary, multiplier-based one as soon as the regularity of 
the considered functions is too poor.

We are now in the position to state the auxiliary problem 
that characterizes the pointwise directional derivatives $\S'(u;h)$
of $\S$ in the situation of \cref{th:dirdiff}.

\begin{theorem}[variational inequality for directional derivatives]%
\label{th:dirdiffVI}%
Consider a fixed control $u \in CBV[0, T]$ with 
associated state $y := \S(u) \in CBV[0, T]$.
Then, for \mbox{every} $h \in CBV[0, T]$, the pointwise directional 
derivative $\delta := \S'(u; h) \in BV[0, T]$ of $\S$ at $u$ in 
direction $h$ is the unique solution in $BV[0, T]$ of the system 
\begin{equation}
\label{eq:dirdiffVI}
\begin{gathered}
\int_0^s (z  - \delta_+)\dd(\delta - h) \geq 0\quad \forall z \in G\left ([0, s]; K^\ptw_{\crit}(y,u) \right)
\quad \forall s \in (0, T],
\\
 \delta_+(t) \in  K^\ptw_{\crit}(y,u)(t)~\forall t \in [0, T], \qquad \delta(0) = 0.
\end{gathered}
\end{equation}
Moreover, it holds $\delta(t) \in \{\delta(t+), \delta(t-)\}$ for all $t \in [0, T]$ and $\var(\delta) \leq 2\var(h)$.
\end{theorem}

\begin{proof}
This follows from \cite[Theorem 2.1]{Brokate2021}, where the result 
is stated for the scalar play operator $\PP(u) := u - \S(u)$.
\end{proof}

As $z=0$ and $z = 2\delta_+$ are admissible test functions in \eqref{eq:dirdiffVI}, this variational inequality 
implies in 
particular that
\begin{equation}
\label{eq:dirdiffeq}
\int_0^s \delta_+\dd(\delta - h) = 0
\quad \forall s \in (0, T].
\end{equation}

We remark that, 
using the inclusion  $\delta(t) \in \{\delta(t+), \delta(t-)\}$ and  \cite[Lemma 6.3.3]{Monteiro2019},
it is easy to check that
the inequality in
\eqref{eq:dirdiffVI} is satisfied by $\delta$
regardless of whether the right limit $\delta_+$ in the integral
is defined w.r.t.\ $[0, s]$ 
or w.r.t.\ $[0, T]$.
To achieve that $\delta$ is uniquely characterized by  \eqref{eq:dirdiffVI},
the definition w.r.t.\ $[0,s]$ and the corresponding 
convention for the endpoint $s$ have to be used, 
see \cite[proof of Theorem 2.1]{Brokate2021}.

Regarding the regularity properties of the
derivatives $\S'(u;h)$ in \cref{th:dirdiffVI}, 
it should be noted that $\S'(u;h)$ can
satisfy $\S'(u;h)_+ \neq \S'(u;h) \neq \S'(u;h)_-$ even when $u$
and $h$ are smooth, see \cite[Example 4.1]{Brokate2021}.
There is, however, a logic behind the 
jumps of $\S'(u;h)$ as the following corollary shows.

\begin{corollary}[direction of jumps]%
\label{corollary:dirdiffjumps}%
Consider the situation in \cref{th:dirdiffVI} for some fixed $u, h \in CBV[0, T]$.
Then, for all $t\in [0,T]$, it holds
\begin{gather}
\label{eq:dirdiffjumps1}
 (\delta(t+)-\delta(t-))\zeta \geq 0 \quad \forall \zeta \in K^\ptw_{\crit}(y,u) (t),
 \\ 
 \label{eq:dirdiffjumps2}
\delta(t+)(\delta(t+)-\delta(t-)) = \delta(t+)(\delta(t+)-\delta(t)) = 0.   
\end{gather}
In particular, if $t \in [0, T]$ is a point of discontinuity of $\delta = \S'(u;h) \in BV[0, T]$,
i.e., if $\delta(t+) \neq \delta(t-)$,
then it holds $\delta(t+) = 0$.
Moreover, we have  $\delta(0+) = \delta(0) = 0$.
\end{corollary}

\begin{proof}
For the test function $z = \mathds{1}_{\{t\}}\zeta$ with $t\in [0,T]$ and $\zeta \in K^\ptw_{\crit}(y,u) (t)$,
we obtain from \eqref{eq:dirdiffVI}, using \eqref{eq:dirdiffeq} as well as \cref{eq:singlepointmass},
\[
0 \leq \int_0^T \mathds{1}_{\{t\}}\zeta\dd(\delta - h) 
= \zeta((\delta - h)(t+) - (\delta - h)(t-))= \zeta(\delta(t+) - \delta(t-))
\]
with the conventions $\delta(0-) = \delta(0)$ and 
$\delta(T+) = \delta(T)$. 
This proves \eqref{eq:dirdiffjumps1}. Using the test functions 
$z = \delta_+ \pm \mathds{1}_{\{t\}}\delta_+(t)$
in \eqref{eq:dirdiffVI}, we obtain analogously
\[
0 \leq \int_0^T \pm \mathds{1}_{\{t\}}\delta_+(t)\dd(\delta - h) 
= \pm \delta(t+)(\delta(t+) - \delta(t-)).
\]
Since $\delta(t) \in \{\delta(t-),\delta(t+)\}$, both equalities in \eqref{eq:dirdiffjumps2} follow. 
All other assertions are immediate consequences of \eqref{eq:dirdiffjumps1}, \eqref{eq:dirdiffjumps2},
and the initial condition $\delta(0) = 0$.
\end{proof}

We would like to point out that jump conditions 
similar to those in \cref{corollary:dirdiffjumps}
also have to be studied in order to establish the 
system \eqref{eq:dirdiffVI}, see 
\cite[section 5]{Brokate2021}. We deduce \cref{corollary:dirdiffjumps}
from \cref{th:dirdiffVI} here to simplify the presentation
and to avoid recalling major parts of the analysis in \cite{Brokate2021}.
As an immediate consequence of 
\cref{th:dirdiffVI,corollary:dirdiffjumps}, we obtain:

\begin{corollary}[variational inequality for the right limits of the derivatives]%
\label{cor:dirdiffVI+}%
Consider an arbitrary but fixed $u \in CBV[0, T]$ with state $y := \S(u) \in CBV[0, T]$. 
Then, for every $h \in CBV[0, T]$,
the right limit $\eta := \S'(u; h)_+ \in BV_r[0, T]$ of the 
pointwise directional derivative $\S'(u; h)$ of $\S$ at $u$ in 
direction $h$ is the unique solution in $BV_r[0, T]$ of the variational inequality  
\begin{equation}
\label{eq:dirdiffVI+}
\begin{gathered}
\int_0^T (z  - \eta)\dd(\eta- h) \geq 0\quad \forall z \in G\left ([0, T]; K^\ptw_{\crit}(y,u) \right),
\\
\eta(t) \in  K^\ptw_{\crit}(y,u)(t)~\forall t \in [0, T], \qquad \eta(0) = 0.
\end{gathered}
\end{equation}
Moreover, for all $s\in (0,T]$, it is true that 
\begin{equation}
\label{eq:dirdiffVI+s}
\int_0^s (z  - \eta)\dd(\eta- h) \geq 0\quad \forall z \in G\left ([0, s]; K^\ptw_{\crit}(y,u) \right).
\end{equation}
\end{corollary}

\begin{proof}
That $\eta$
satisfies the second line of \eqref{eq:dirdiffVI+} 
follows from \cref{th:dirdiffVI} and \cref{corollary:dirdiffjumps}.
Since $\S'(u;h)\in BV[0,T]$ has at most countably many discontinuity points
by \cite[Theorem 2.3.2]{Monteiro2019}, and because $(\eta - \S'(u;h))(T) = 0$ by
convention and $(\eta - \S'(u;h))(0) = 0$ by \cref{corollary:dirdiffjumps}, it follows from \cref{lemma:intgzeroexcept}
that
 \[
 \int_0^T f\dd(\eta- \S'(u;h)) = 0\quad \forall f \in G[0, T].
 \]
If we combine this identity with \eqref{eq:dirdiffVI} 
for $s = T$ and the linearity of 
the Kurzweil-Stieltjes integral, then the 
variational inequality in \eqref{eq:dirdiffVI+}
follows immediately. 
To establish \eqref{eq:dirdiffVI+s},
it suffices to consider functions of the form 
$z := \mathds{1}_{[0,s]}\tilde z + \mathds{1}_{(s,T]}\eta$,
$s \in (0, T]$, $\smash{\tilde z \in  G\left ([0, s]; K^\ptw_{\crit}(y,u) \right )}$, in \eqref{eq:dirdiffVI+}
and to exploit \eqref{eq:decomposeInterval}
and \eqref{eq:singlepointmass}.

Suppose now that there are two 
$\eta_1, \eta_2 \in BV_r[0,T]$ satisfying \eqref{eq:dirdiffVI+}. In this case, we 
can consider functions of the form
$z := \mathds{1}_{[0,s]}\eta_2 + \mathds{1}_{(s,T]}\eta_1$
and
$z := \mathds{1}_{[0,s]}\eta_1 + \mathds{1}_{(s,T]}\eta_2$
in the inequalities for $\eta_1$
and $\eta_2$, respectively, 
and add the resulting estimates to obtain
with 
\eqref{eq:decomposeInterval} and \eqref{eq:singlepointmass}
that 
$
\int_0^s (\eta_2  - \eta_1)\dd(\eta_2 - \eta_1) \le 0
$ holds for all $s\in (0,T]$. 
Due to
\cref{prop:partialIntegration} and 
$\eta_1(0) = \eta_2(0) = 0$, this yields 
$(\eta_1(s) - \eta_2(s))^2 \leq 0$ 
for all $s \in [0,T]$. This proves 
that \eqref{eq:dirdiffVI+} 
possesses at most one 
solution in $BV_r[0,T]$.
\end{proof}

Note that the system \eqref{eq:dirdiffVI+} 
has the same structure as ``usual'' rate-independent systems 
posed in $BV_r[0,T]$, cf.\ \cite[Theorem 3.3]{Recupero2020}. 
Because of this, \eqref{eq:dirdiffVI+} is easier to work with than 
\eqref{eq:dirdiffVI}, which involves the additional varying parameter $s \in (0,T]$.

\section{First consequences for the optimal control problem (P)}\label{sec:5}

As a direct consequence of the results
for $\S$ in the last section, we 
obtain:

\begin{corollary}[existence of solutions]%
\label{cor:solex}%
Assume, in addition to the conditions 
in our standing \cref{ass:standing},
that:
\begin{itemize}
\item $(U, \|\cdot\|_U)$ is a reflexive Banach space that is compactly embedded into $C[0,T]$,
\item $\Uad$ is a closed subset of $(U, \|\cdot\|_U)$,
\item $\JJ$ is lower semicontinuous 
in the sense that, for 
all $\{(y_n, z_n, u_n) \} \subset C[0, T] \times \R \times U$
satisfying $y_n \to y$ in $C[0, T]$, $z_n \to z$ in $\R$, and $u_n \weakly u$ in $U$,
we have 
\[
\liminf_{n \to \infty} \JJ(y_n, z_n, u_n) \geq \JJ(y,z,u),
\]
\item $\JJ$ is radially unbounded in the sense that there exists a function $\rho\colon [0, \infty) \to \R$
satisfying $\rho(s) \to \infty$ for $s \to \infty$ and 
\[
\JJ(y, z, u) \geq \rho\left ( \|u\|_U \right ) \qquad \forall (y, z, u) \in C[0, T] \times \R \times U.
\]
\end{itemize}
Then the problem \eqref{eq:P} possesses at least one globally optimal control-state pair $(\bar u, \bar y)$.
\end{corollary}\pagebreak

\begin{proof}
This follows straightforwardly 
from the direct method of the calculus of variations
and the Lipschitz continuity property in 
\eqref{eq;LipschitzSinfty}.
\end{proof}

A prototypical example of a space $U$ satisfying the 
conditions in \cref{cor:solex} is $H^1(0,T)$. 
We would like 
to point out 
that it is, in general, not possible to use the 
direct method of the calculus of variations in the 
situation of \cref{cor:solex} if 
the control space $U$ is not compactly embedded 
into $C[0,T]$ and 
if the convergence
$u_n \weakly u$ in $U$
only implies \smash{$u_n \weaklystar u$} in $BV[0,T]$.
To see this, suppose that 
$r = y_0 = 1$, that $T = 2$, and 
that $\varphi \in C^\infty(\R)$ is a function that 
is identical zero in $\R \setminus (0,2)$, equal to 2 at $t=1$,
monotonously increasing in $[0,1]$, and monotonously 
decreasing in $[1,2]$.
For such $r$, $y_0$, $T$, and $\varphi$,
it is easy to check that the controls 
$u_n(t) := \varphi(n t)$, $t \in [0,T]$, $n \in \mathbb{N}$,
satisfy $u_n \in C^\infty[0,T]$, 
$\|u_n\|_{BV} = \var(u_n) = 4$, 
and
$\S(u_n) = \mathds{1}_{[0, 1/n)} + \mathds{1}_{[1/n, T]} (u_n - 1)$
for all $n$
as well as $u_n(t) \to 0$ for all $t \in [0,T]$ and $n \to \infty$.
In particular, we have
$\|\S(u_n)\|_{BV} = 1 + \var(\S(u_n)) = 3$
for all $n$
and $\S(u_n)(t) \to \mathds{1}_{\{0\}}(t) - \mathds{1}_{(0, T]}(t)$
for all $t \in [0,T]$ and $n \to \infty$. 
In view of \cite[Proposition~3.13]{Ambrosio2000},
this yields
\smash{$C^\infty[0,T] \ni u_n \weaklystar 0$} 
and 
\smash{$CBV[0,T] \ni \S(u_n) \weaklystar 
\mathds{1}_{\{0\}} - \mathds{1}_{(0, T]} \neq \mathds{1}_{[0, T]} = \S(0)$}
in $BV[0,T]$.
The map $\S$ is thus not continuous 
w.r.t.\ weak-star convergence in $BV[0,T]$ -- even along sequences 
of smooth functions -- and we may conclude that
it is indeed not possible to apply the direct method of the 
calculus of variations to establish the solvability of 
\eqref{eq:P} if the space $U$ only provides
weak-star convergence in $BV[0,T]$ for minimizing 
sequences. We remark that
the compact embedding $U \hookrightarrow C[0,T]$
needed in \cref{cor:solex} 
significantly complicates the 
derivation of the strong stationarity 
system \eqref{eq:strongstatsys-2}
since it makes it impossible to 
find sequences that converge weakly or strongly in $U$
to the discontinuous directional derivatives $\S'(u;h)$. 
In fact, this difficulty already arises due to the 
embedding $U \hookrightarrow C[0,T]$ in \cref{ass:standing}.
We will circumvent this problem in \cref{sec:6} by means of a 
careful analysis of pointwise limits. The next 
corollary is concerned with the Bouligand stationarity condition 
that arises from \cref{th:dirdiff}.

\begin{corollary}[Bouligand stationarity condition]%
Suppose that $\bar u \in \Uad$ is a locally optimal control of  \eqref{eq:P} with associated state $\bar y := \S(\bar u)$.
Then it holds 
\begin{equation}
\label{eq:Bouligand}
\begin{aligned}
\left \langle \partial_1 \JJ(\bar y, \bar y(T), \bar u), \S'(\bar u; h) \right \rangle_{L^\infty}
&+
\partial_2 \JJ(\bar y, \bar y(T), \bar u)\S'(\bar u; h)(T)
\\
&+
\left \langle \partial_3 \JJ(\bar y, \bar y(T), \bar u), h \right \rangle_{U}
\geq 
0
\quad 
\forall h \in \R_+(\Uad - \bar u). 
\end{aligned}
\end{equation}
Here, $ \partial_1 \JJ(\bar y, \bar y(T), \bar u)\in L^1(0, T)$,  
$\partial_2 \JJ(\bar y, \bar y(T), \bar u) \in \R$, 
and $\partial_3 \JJ(\bar y, \bar y(T), \bar u) \in U^*$
are the partial Fr\'{e}chet derivatives of the objective function
$\JJ\colon L^\infty(0, T) \times \R \times U \to \R$.
\end{corollary}

\begin{proof}
This 
follows along standard lines from the convexity of 
$\Uad$,
the Fr\'{e}chet differentiability of $\JJ$, 
 \cref{th:dirdiff}, 
the Lipschitz estimate \eqref{eq;LipschitzSinfty},
and
the $L^1$-regularity of $ \partial_1 \JJ(\bar y, \bar y(T), \bar u)$. See, e.g., 
\cite[Proposition 6.1.2]{ChristofPhd2018}
or 
\cite[section 3]{Herzog2013}.
\end{proof}

The last result motivates:

\begin{definition}[Bouligand stationary point]%
\label{def:Bouligandstationary}%
A control $\bar u \in \Uad$ with 
associated state $\bar y := \S(\bar u)$
is called a Bouligand stationary point of \eqref{eq:P}
if $(\bar u,\bar y)$ satisfies \eqref{eq:Bouligand}.
\end{definition}

Due to its implicit nature, the Bouligand stationarity condition \eqref{eq:Bouligand} 
is typically not very helpful in practice.
This is one of the main motivations for the derivation of 
strong stationarity systems. 
To establish such a system for \eqref{eq:P}, we study:

\section{Temporal polyhedricity properties}
\label{sec:6}
Throughout this section, 
we assume that an arbitrary but fixed $u \in CBV[0,T]$
with state $y := \S(u) \in CBV[0,T]$ is given.
For these $u$ and $y$, we introduce:

\begin{definition}[reduced critical cone and smooth critical radial directions]%
\label{def:6.1}
We define the reduced critical cone in $G_r[0, T]$ associated with
$(y,u)$ to be the set 
\begin{equation*}
\begin{aligned}
\KK_{G_r}^{\red,\crit}(y,u)  
:=
\big \{
z \in G_r[0, T]
\colon
&z(t) \in K_{\crit}^{\ptw}(y,u)(t)\,\forall t \in [0,T],~ z(0) = 0,
\\
&\text{and } z(t) = 0~\forall t \in [0,T] \text{ with } z(t-) \neq z(t)
\big\}
\end{aligned}
\end{equation*}
and the cone of smooth critical radial directions associated with $(y,u)$ to be the set
\begin{equation*}
\begin{aligned}
\KK_{C^\infty}^{\rad,\crit}(y,u) 
:=
\big\{
z \in C^\infty[0, T]
\colon
&z(t) \in K^\ptw_{\crit}(y,u)(t)~\forall t \in [0,T], \, z(0) = 0, 
\\
&\text{and } \exists \alpha > 0 \text{ s.t. } y(t) + \alpha z(t) \in Z~\forall t \in [0, T] 
\big\}.
\end{aligned}
\end{equation*}
\end{definition}

Note that 
$\smash{\KK_{C^\infty}^{\rad,\crit}(y,u)}$
is a subset of 
$\smash{\KK_{G_r}^{\red,\crit}(y,u)}$,
that both $\smash{\KK_{C^\infty}^{\rad,\crit}(y,u)}$
and $\smash{\KK_{G_r}^{\red,\crit}(y,u)}$ are cones
containing the zero function,
and that $\smash{\KK_{C^\infty}^{\rad,\crit}(y,u)}$
is convex. The cone
\smash{$\KK_{G_r}^{\red,\crit}(y,u)$} is typically not convex 
due to the additional conditions on the points of discontinuity.
From \cref{corollary:dirdiffjumps,cor:dirdiffVI+},
it follows that 
$\S'(u;h)_+$ is an element of 
\smash{$\KK_{G_r}^{\red,\crit}(y,u)$}
for all $h \in CBV[0,T]$. 
In fact, $\smash{\KK_{G_r}^{\red,\crit}(y,u)}$
collects all information 
about the pointwise properties of the right limits of the
derivatives $\S'(u;h)$
that we have derived so far. 
This motivates the name ``reduced critical cone'', cf.\ the analysis 
for elliptic variational inequalities
in \cite{ChristofPhd2018}.
From
\cref{prop:classicalcritical},
we obtain that
\begin{align*}
\KK_{G_r}^{\red,\crit}(y,u)  
&=
\Bigg\{
z \in G_r[0, T]
\colon
 z(t) \in K_{\rad}^{\ptw}(y)(t)\,\forall t \in [0,T], \int_0^T z \dd(y-u) = 0,
\\[-0.1cm]
&\hspace{2.8cm}z(0) = 0, z(t) = 0~\forall t \in [0,T] \text{ with } z(t-) \neq z(t)
\Bigg\}
\end{align*}
and
\begin{align*}
\KK_{C^\infty}^{\rad,\crit}(y,u) 
&=
\Bigg \{
z \in C^\infty[0, T]
\colon
z(0) = 0, ~ \int_{0}^{T} z \dd(y-u)  = 0, \text{ and}
\\[-0.1cm]
&\hspace{2.9cm}\exists \alpha > 0 \text{ s.t. } y(t) + \alpha z(t) \in Z~\forall t \in [0, T] 
\Bigg \}.
\end{align*}

The main result of this section 
--
\cref{theorem:tempoly}
-- shows that the 
cone $\smash{\KK_{C^\infty}^{\rad,\crit}(y,u)}$ 
is, in a suitably defined sense, dense in
$\smash{\KK_{G_r}^{\red,\crit}(y,u)}$. 
This density property extends the concept of polyhedricity to
the setting considered in this paper.
In the case of the elliptic problem \eqref{eq:optstaclecontrol}, 
polyhedricity expresses that the set
$ K_{\rad}(y) \cap (u + \Delta y)^\perp$
is $H_0^1(\Omega)$-dense in the critical cone
$K_{\tan}(y) \cap (u + \Delta y)^\perp$,
see \cite{Haraux1977,Wachsmuth2019}. 
For the study 
of the inequality \eqref{eq:V}, the set \smash{$\KK_{C^\infty}^{\rad,\crit}(y,u)$}
is relevant because of the following observation. 

\begin{lemma}[directional derivative in smooth critical radial directions]%
\label{lemma:dirdifcritrad}%
Let $h$ be an arbitrary but fixed element of the set \smash{$\KK_{C^\infty}^{\rad,\crit}(y,u)$}.
Then there exists $\alpha > 0$ such that $\S(u + \beta h) = \S(u) + \beta h$
holds for all $\beta \in (0, \alpha)$. In particular,  $\S'(u;h) = h$. 
\end{lemma}

\begin{proof}
According to the definition of the set $\KK_{C^\infty}^{\rad,\crit}(y,u) $,
we can find a number $\alpha > 0$ such that $y(t) + \alpha h(t) \in Z$ holds for all $t \in [0, T]$.
Since $Z$ is convex, this also yields $y(t) + \beta h(t) \in Z$  for all $t \in [0, T]$ and all $\beta \in (0, \alpha)$.
From \cref{prop:classicalcritical} and the 
variational inequality \eqref{eq:V} for $y$, 
we moreover obtain that 
\[
\int_0^T h\dd(y-u) = 0
\qquad
\text{and}
\qquad
\int_0^T (v - y)\dd (y - u)\geq 0 ~~\forall v \in C([0, T]; Z).
\]
If we combine the above with the initial conditions $y(0) = y_0$ and $h(0) = 0$
and our previous considerations, then it follows that 
\begin{equation*}
\begin{aligned}
& \int_0^T (v - (y + \beta h) )\dd (y + \beta h - (u + \beta h))\geq 0 &&\forall v \in C([0, T]; Z),
\\
&y(t) + \beta h(t) \in Z\quad \forall t \in [0, T], &&y(0)+ \beta h(0) = y_0,
\end{aligned}
\end{equation*}
holds for all $\beta \in (0, \alpha)$. Thus,
$\S(u + \beta h) = y + \beta h$
for all $\beta \in (0, \alpha)$ by \cref{th:Swellposed}  as claimed. The assertion about the directional derivative follows immediately 
from this identity. This completes the proof. 
\end{proof}

Note that 
\cref{lemma:dirdifcritrad} remains valid when 
the space $C^\infty[0, T]$
in the definition of  \smash{$\KK_{C^\infty}^{\rad,\crit}(y,u)$}
is replaced with the space $CBV[0,T]$. We consider 
smooth critical radial directions in our analysis 
because this gives rise to a stronger density result
in \cref{theorem:tempoly}.
As we will see in \cref{sec:7}, 
\cref{lemma:dirdifcritrad}
makes it possible to prove the strong stationarity 
system \eqref{eq:strongstatsys-2}
once the polyhedricity  property in \cref{theorem:tempoly}
is established. 
To  obtain the latter, 
we require the following result.

\begin{lemma}%
\label{lemma:polyprep1}%
Suppose that $z \in \KK_{G_r}^{\red,\crit}(y,u)$ and $\xi > 0$ are given.
Let $t \in [0, T]$ be an arbitrary but fixed point of continuity of $z$, i.e., 
a point with $z(t) = z(t-)$. 
Then there exists 
$\varepsilon > 0$ such that the step function
\[
\zeta \colon [0,T] \to \R, \quad \zeta(s) := z(t)\mathds{1}_{J_\varepsilon(t)}(s), \quad 
J_\varepsilon(t) := [t-\varepsilon,t+\varepsilon]\cap [0,T],
\]
possesses all of the following properties:
\begin{enumerate}[label=\roman*)]
\item\label{lemma:polyprep1:i} It is true that 
\[
\sup_{s \in [t-\varepsilon,t+\varepsilon]\cap [0,T]} 
\left |
z(s) - \zeta(s)
\right | \leq \xi.
\]
\item\label{lemma:polyprep1:ii}
 It holds 
\[
\zeta(s) 
\in
K^\ptw_{\crit}(y,u)(s)   \qquad \forall s \in [0,T].
\]
\item\label{lemma:polyprep1:iii}
 For every $0\leq \psi \in C_c^\infty(\R)$ 
 with support 
 $\supp(\psi) \subset (t-\varepsilon, t + \varepsilon)$, 
the function $\psi \zeta \in G[0,T]$ is an element of the cone $\KK_{C^\infty}^{\rad,\crit}(y,u)$. 
\end{enumerate}
\end{lemma}

\begin{proof}
Since $z$ is continuous at $t$, 
we can find $\varepsilon > 0$
such that \ref{lemma:polyprep1:i} holds. 
If $z(t) = 0$, then $\zeta = 0$ and 
\ref{lemma:polyprep1:ii}
and
\ref{lemma:polyprep1:iii}
hold trivially for this $\varepsilon$. 
Due
to the definition of the set \smash{$\KK_{G_r}^{\red,\crit}(y,u)$}
and the continuity of $z$ at $t$,
this case covers in particular 
the situations $t = 0$ and $t\in \closure(A(y,u))$. 
In what follows, we may thus  assume that 
\begin{equation}
\label{eq:randomeq3746}
z(t) \neq 0 \qquad \text{and}\qquad 
J_\varepsilon(t) \subset \big ( I(y) \cup B(y,u) \big ) 
\cap (0, T]
\end{equation}
and have to prove that, 
for a potentially smaller $\varepsilon$,
we have 
\ref{lemma:polyprep1:ii}
and 
\ref{lemma:polyprep1:iii}.
To this end, we distinguish between three cases.

Case 1: $t\in I(y)$. 
In this case, it follows from 
the continuity of $y$ that, 
after possibly making $\varepsilon$ smaller,
we have $J_\varepsilon(t) \subset I(y)$
and $|y| \le r-\gamma$ on $J_\varepsilon(t)$
for some $\gamma > 0$.
This implies in particular that 
$K^\ptw_{\crit}(y,u)(s)= \R$  for all $s\in J_\varepsilon(t)$.

Case 2: $t\in B_+(y,u)$.
In this case,
it follows from the continuity of $y$
that, after possibly making $\varepsilon$ smaller, 
we have $J_\varepsilon(t) \subset I(y)\cup B_+(y,u)$ and
$y \ge -r+\gamma$ on $J_\varepsilon(t)$ 
for some $\gamma > 0$.
Due to the definition of 
 $K^\ptw_{\crit}(y,u)$, 
this implies 
in particular that
$z(t) \in K^\ptw_{\crit}(y,u)(t)= (-\infty,0]
\subset K^\ptw_{\crit}(y,u)(s)$ for all 
$s \in J_\varepsilon(t)$. 

Case 3: $t\in B_-(y,u)$.
In this case,
it follows from the continuity of $y$
that, after possibly making $\varepsilon$ smaller, 
we have $J_\varepsilon(t) \subset I(y)\cup B_-(y,u)$ and 
$y \le r-\gamma$ on $J_\varepsilon(t)$ for some
$\gamma > 0$.
Due to the definition of 
 $K^\ptw_{\crit}(y,u)$, 
this implies 
in particular that
$z(t) \in K^\ptw_{\crit}(y,u)(t)= [0, \infty)
\subset K^\ptw_{\crit}(y,u)(s)$ for all 
$s \in J_\varepsilon(t)$. 

In all of the above cases, 
the resulting $\varepsilon > 0$ 
satisfies 
$z(t) = \zeta(s) 
\in
K^\ptw_{\crit}(y,u)(s)$
and 
$(y + \alpha \zeta)(s) \in Z$
for all $s\in J_\varepsilon(t)$
and all
$0 < \alpha \leq \gamma \|\zeta\|_\infty^{-1}$.
Since $\zeta(s) = 0$ 
for $s\notin J_\varepsilon(t)$,
these inclusions for $\zeta$ are also true for 
all $s \in [0,T]$. 
This proves \ref{lemma:polyprep1:ii}. 
Consider now a function $0 \leq \psi\in C_c^\infty(\R)$ 
with $\supp(\psi)\subset (t-\varepsilon,t+\varepsilon)$. 
Then $\psi\zeta\in C^\infty[0,T]$ 
and it follows from the nonnegativity of $\psi$, 
the properties of $\zeta$, 
the cone property of $K^\ptw_{\crit}(y,u)(s)$,
and \eqref{eq:randomeq3746} that 
$(\psi\zeta)(0) = 0$ holds and that 
$(\psi\zeta)(s)\in K^\ptw_{\crit}(y,u)(s)$
and $(y+\alpha\psi\zeta)(s)\in Z$ 
for all $s\in [0,T]$
and all $0 < \alpha \leq \gamma \|\psi \|_\infty^{-1}\|\zeta\|_\infty^{-1}$.
This shows \smash{$\psi \zeta \in \KK_{C^\infty}^{\rad,\crit}(y,u)$},
establishes \ref{lemma:polyprep1:iii}, and
completes the proof. 
\end{proof}

The next lemma is a version of \cref{lemma:polyprep1} 
for points of discontinuity. 

\begin{lemma}%
\label{lemma:polyprep2}%
Suppose that $z \in \KK_{G_r}^{\red,\crit}(y,u)$ and $\xi > 0$ are given.
Let $t \in [0, T]$ be an arbitrary but fixed point of discontinuity of $z$, i.e., 
a point with $z(t) \neq z(t-)$. 
Then there exists 
$\varepsilon > 0$ such that the step function
\[
\zeta \colon [0,T] \to \R, \qquad 
\zeta (s) := z(t-)\mathds{1}_{J_\varepsilon^-(t)}(s), \quad
J_\varepsilon^-(t) := [t-\varepsilon,t) \cap [0,T],
\]
possesses the following properties:
\begin{enumerate}[label=\roman*)]
\item\label{lemma:polyprep2:i}
 It is true that 
\[
\sup_{s \in [t-\varepsilon, t + \varepsilon] \cap [0,T]} 
\left |
z(s) - \zeta(s)
\right | \leq \xi.
\]
\item\label{lemma:polyprep2:ii} 
It holds 
\[
\zeta(s) 
\in
K^\ptw_{\crit}(y,u)(s)   \qquad \forall s \in [0,T].
\]
\item\label{lemma:polyprep2:iii} 
For every $0 \leq \psi \in C_c^\infty(\R)$
with support $\supp(\psi) \subset (t-\varepsilon, t)$, 
the function $\psi \zeta \in G[0,T]$ is an element of the cone $\KK_{C^\infty}^{\rad,\crit}(y,u)$. 
\end{enumerate}
\end{lemma}

\begin{proof}
Since $z\in\KK_{G_r}^{\red,\crit}(y,u)$,
it necessarily holds $t > 0$ and $z(t) = 0$.
As $z$ is right-continuous, 
this implies that there exists $\varepsilon > 0$
such that 
\ref{lemma:polyprep2:i} is satisfied.
Moreover, $z(t-) \neq 0$ because $z$ is assumed to be discontinuous at $t$. Since $z = 0$ on $ A(y,u)$,
it follows that, 
for a potentially smaller $\varepsilon$, we have 
\begin{equation}\label{eq:polyprep2}
J_\varepsilon^-(t) 
 \subset \big ( I(y) \cup B(y,u) \big ) 
\cap (0, T].
\end{equation}
We now again distinguish between three cases.

Case 1: After possibly making $\varepsilon$
smaller, we have $J_\varepsilon^-(t)\subset I(y)$.
In this case, 
it holds $K^\ptw_{\crit}(y,u)(s)= \R$ for all $s\in J_\varepsilon^-(t)$ and it follows from the 
continuity of $y$ that, for every 
compact set $E \subset J_\varepsilon^-(t)$,
we can find a number $\gamma > 0$ with 
$|y| \le r-\gamma$ on $E$.

Case 2: There exists a sequence $\{s_n\}\subset B_+(y,u)$ with $s_n\to t^-$. 
In this case, we have $y(s_n) = r$ and $z(s_n)\in K^\ptw_{\crit}(y,u)(s_n)= (-\infty,0]$ for all $n$
and it follows that $y(t) = r$ and $z(t-)\le 0$. 
Due to the continuity of $y$ and \eqref{eq:polyprep2},
this implies that, after possibly making 
$\varepsilon$ smaller, we have
$J_\varepsilon^-(t)  \subset I(y) \cup B_+(y,u)$.
In particular, it holds 
$z(t-) \in (-\infty,0] \subset K^\ptw_{\crit}(y,u)(s)$ for all $s\in J_\varepsilon^-(t)$ and, 
for all compact $E \subset J_\varepsilon^-(t)$,
we can find a number $\gamma > 0$ with 
$y \ge -r + \gamma$ on $E$.

Case 3: There exists a sequence $\{s_n\}\subset B_-(y,u)$ with $s_n\to t^-$. 
In this case, we have $y(s_n) = -r$ and $z(s_n)\in K^\ptw_{\crit}(y,u)(s_n)= [0, \infty)$ for all $n$
and it follows that $y(t) = -r$ and $z(t-)\ge 0$. 
Due to the continuity of $y$ and \eqref{eq:polyprep2},
this implies that, after possibly making 
$\varepsilon$ smaller, we have
$J_\varepsilon^-(t)  \subset I(y) \cup B_-(y,u)$.
In particular, it holds 
$z(t-) \in [0, \infty) \subset K^\ptw_{\crit}(y,u)(s)$ for all $s\in J_\varepsilon^-(t)$ and, 
for all compact $E \subset J_\varepsilon^-(t)$,
we can find a number $\gamma > 0$ with 
$y \le r - \gamma$ on $E$.

In all of the above cases, 
the resulting $\varepsilon > 0$ 
satisfies $z(t-) = \zeta(s) 
\in
K^\ptw_{\crit}(y,u)(s)$
for all $s\in J_\varepsilon^-(t)$.
Since $\zeta(s) = 0$ 
for $s\notin J_\varepsilon^-(t)$, this proves \ref{lemma:polyprep1:ii}. Moreover, 
we obtain from the above construction
that, for every compact set $E \subset J_\varepsilon^-(t)$, there exists a number $\gamma > 0$
with $(y + \alpha \zeta)(s) \in Z$
for all $s \in E$ and all 
$0 < \alpha \leq \gamma \|\zeta\|_\infty^{-1}$.
If a function $\psi\in C_c^\infty(\R)$ with $\psi\ge 0$ and 
support $E := \supp(\psi)\subset (t-\varepsilon,t)$ 
is given, 
then this implies that
$(y+\alpha\psi\zeta)(s)\in Z$ 
holds for all $s\in [0,T]$ and all 
$0 < \alpha \leq \gamma \|\psi\|_\infty^{-1}\|\zeta\|_\infty^{-1}$. 
Due to the nonnegativity of $\psi$,
the properties of $\zeta$,
and the cone property of $K^\ptw_{\crit}(y,u)(s)$,
one further obtains that 
$(\psi\zeta)(s)\in K^\ptw_{\crit}(y,u)(s)$ holds 
for all $s\in [0,T]$, and due to \eqref{eq:polyprep2}
and the properties of $\supp(\psi)$,
that $(\psi\zeta)(0) = 0$
and $\psi \zeta \in C^\infty[0,T]$.
Thus,
\smash{$\psi \zeta \in \KK_{C^\infty}^{\rad,\crit}(y,u)$}.
This establishes 
\ref{lemma:polyprep2:iii} and completes 
the proof. 
\end{proof}

We can now prove the main result of this section.

\begin{theorem}[temporal polyhedricity]%
\label{theorem:tempoly}%
Let $z \in \KK_{G_r}^{\red,\crit}(y,u)$ be given.
Then there exist functions $z_{i,j}, z_j \in G_r[0,T]$, $i,j \in \mathbb{N}$, such that the following is true:
\begin{equation*}
\begin{gathered}
z_{i,j} \in \KK_{C^\infty}^{\rad,\crit}(y,u), 
\qquad \|z_{i,j}\|_\infty \leq \|z\|_\infty~\forall i,j,
\\
z_{j} \in \KK_{G_r}^{\red,\crit}(y,u),
\qquad \|z_{j}\|_\infty \leq \|z\|_\infty~\forall j,
\\
z_{i,j} \to z_j \text{ pointwise in } [0,T] \text{ for } i \to \infty \text{ for all }j,
\\
z_j \to z \text{ uniformly in } [0,T] \text{ for } j \to \infty.
\end{gathered}
\end{equation*}
\end{theorem}

\begin{proof}
Consider an arbitrary but fixed $j \in \mathbb{N}$ and define $\xi := 1/j$. 
For every $t \in [0,T]$, we choose $\varepsilon_t > 0$ for this $\xi$ as in \cref{lemma:polyprep1,lemma:polyprep2}.
This results in a collection of open intervals $(t - \varepsilon_t, t + \varepsilon_t)$ that 
covers $[0,T]$. By compactness, we can choose a finite subcover of this collection.
We denote the time points of this cover with $t_k$, $k=1,...,N$, $N \in \mathbb{N}$, 
and the associated $\varepsilon_{t_k}$ with $\varepsilon_k$, $k=1,...,N$. 
We assume w.l.o.g.\ that there are no $k,l$ satisfying 
$(t_k - \varepsilon_k, t_k + \varepsilon_k) \subset (t_l - \varepsilon_l, t_l + \varepsilon_l)$
and $k \neq l$. In this case, 
by possibly 
making the intervals  $(t_k - \varepsilon_k, t_k + \varepsilon_k)$
smaller, we can construct intervals $(t_k - a_k, t_k + b_k)$, $a_k, b_k > 0$, such that 
\begin{equation*}
\begin{gathered}
(t_k - a_k, t_k + b_k) \subset (t_k - \varepsilon_k, t_k + \varepsilon_k) ~\forall k = 1,...,N,
\qquad
[0,T] \subset \bigcup_{k=1}^N (t_k - a_k, t_k + b_k),
\\
\text{and } t_k \notin  (t_l - a_l, t_l+ b_l)~\forall k \neq l.
\end{gathered}
\end{equation*}
Consider now a smooth partition of unity on $[0,T]$ 
subordinate to the modified cover $(t_k - a_k, t_k + b_k)$,
$k=1,...,N$,
i.e., a collection of functions $\psi_k$, $k=1,...,N$, satisfying 
\begin{equation*}
\begin{gathered}
\psi_k \in C_c^\infty(\R),\quad 0 \leq \psi_k(t) \leq 1~\forall t \in \R,
\quad 
\supp(\psi_k) \subset (t_k - a_k, t_k + b_k) ~\forall k=1,...,N,
\\
\sum_{k=1}^N \psi_k(t) = 1~\forall t \in [0,T],
\end{gathered}
\end{equation*}
see, e.g., \cite{Evans2010},
and choose an arbitrary but fixed function $\varphi \in C^\infty(\R)$ satisfying 
\[
0 \leq \varphi(t) \leq 1~\forall t \in \R,\quad 
\varphi(t) = 1~\forall t \in (-\infty, -1],\quad 
\varphi(t) = 0~\forall t \in [0,\infty).
\]
Define 
\begin{equation*}
z_{i,j}(s)
:=
\sum_{k \colon z(t_k)= z(t_k-)} z(t_k) \psi_k(s)
+
\sum_{k \colon z(t_k) \neq z(t_k-)} z(t_k-) \psi_k(s) \varphi\left ( \frac{s - t_k + 1/i}{1/i} \right )
\end{equation*}
for all $i \in \mathbb{N}$ and $s \in [0,T]$. We claim that $z_{i,j} \in \KK_{C^\infty}^{\rad,\crit}(y,u)$ holds 
for all $i \in \mathbb{N}$. To see this, we first note that we have 
\[
z(t_k) \psi_k(\cdot)\big |_{[0,T]} \in \KK_{C^\infty}^{\rad,\crit}(y,u)\quad \forall k\colon z(t_k)= z(t_k-)
\]
by \cref{lemma:polyprep1}\ref{lemma:polyprep1:iii} and the 
condition $\supp(\psi_k) \subset (t_k - a_k, t_k + b_k) \subset (t_k - \varepsilon_k, t_k + \varepsilon_k)$ for all $k$. 
Analogously, we also have 
\[
z(t_k-) \psi_k(\cdot) \varphi\left ( \frac{\cdot - t_k + 1/i}{1/i} \right )\Bigg |_{[0,T]}  \in \KK_{C^\infty}^{\rad,\crit}(y,u)\quad \forall k\colon z(t_k)\neq z(t_k-)
\]
by the properties of $\psi_k$ and $\varphi$ and \cref{lemma:polyprep2}\ref{lemma:polyprep2:iii}. 
By combining these facts with the observation 
that $\smash{ \KK_{C^\infty}^{\rad,\crit}(y,u)}$ is a convex cone, 
the inclusion $z_{i,j} \in \KK_{C^\infty}^{\rad,\crit}(y,u)$ follows immediately.
Due to the properties of $\varphi$, we further have 
\[
z_{i,j}(s) 
\to 
\sum_{k \colon z(t_k)= z(t_k-)} z(t_k) \psi_k(s)
+
\sum_{k \colon z(t_k) \neq z(t_k-)} z(t_k-) \psi_k(s) \mathds{1}_{(-\infty, t_k)}(s)
\]
for all $s \in [0,T]$ for $i \to \infty$.
Let us denote the function on the right of the last limit with $z_j$. 
By construction, 
the points of discontinuity of this function $z_j$
are precisely the points $t_k$ with $z(t_k) \neq z(t_k-)$. 
Further, at these points, the function $z_j$ is clearly right-continuous and,
by the choice of the functions $\psi_k$ and the 
condition $ t_k \notin  (t_l - a_l, t_l+ b_l)$ for all $k \neq l$, we have 
\[
z_j(t_k) = z(t_k-) \psi_k(t_k) \mathds{1}_{(-\infty, t_k)}(t_k) = 0 
\]
for all $k$ with $z(t_k) \neq z(t_k-)$. 
In combination with the choice of the functions $\psi_k$, 
this yields
$z_j \in G_r[0,T]$, $z_j(t) = z_j(t+) = 0$ for all $t \in [0,T]$ with $z_j(t) \neq z_j(t-)$,
and $z_j(0) = 0$. Due to the properties of $\psi_k$, the 
inclusion $(t_k - a_k, t_k + b_k) \subset (t_k - \varepsilon_k, t_k + \varepsilon_k)$ for all $k$,
the second points of \cref{lemma:polyprep1,lemma:polyprep2},
and the fact that $K^\ptw_{\crit}(y,u)(s)$ is a convex cone for all $s \in [0,T]$,
we also have $\smash{z_j(s) \in K^\ptw_{\crit}(y,u)(s)}$ for all $s \in [0,T]$.
In summary, this allows us to conclude that $z_j \in \KK_{G_r}^{\red,\crit}(y,u)$
holds as desired. It remains to establish the uniform convergence of $z_j$ to $z$ for $j \to \infty$.
To this end, we note that, due to the properties of the partition of 
unity $\{\psi_k\}$, we have 
\begin{equation*}
\begin{aligned}
&\sup_{s \in [0,T]}\left |
z(s) - z_j(s)
\right |
\\
&=
\sup_{s \in [0,T]}\left |
z(s) -
\sum_{k \colon z(t_k)= z(t_k-)} z(t_k) \psi_k(s)
-
\sum_{k \colon z(t_k) \neq z(t_k-)} z(t_k-) \psi_k(s) \mathds{1}_{(-\infty, t_k)}(s)
\right |
\\
&=
\sup_{s \in [0,T]}\left |
\sum_{k \colon z(t_k)= z(t_k-)} \left (z(s) - z(t_k) \right ) \psi_k(s) \right.
\\
&\qquad\qquad\quad \left. + 
\sum_{k \colon z(t_k) \neq z(t_k-)} \Big( z(s) - z(t_k-) \mathds{1}_{(-\infty, t_k)}(s)\Big )\psi_k(s)
\right |
\\
&\leq 
\sup_{s \in [0,T]} 
\left ( 
\sum_{k \colon z(t_k)= z(t_k-)} 
\sup_{\tau \in [t_k - \varepsilon_k, t_k + \varepsilon_k] \cap [0,T]}
\left | z(\tau) - z(t_k) \right | \psi_k(s)  \right.
\\
&\qquad\qquad\quad  + \left.
\sum_{k \colon z(t_k) \neq z(t_k-)} 
\sup_{\tau \in [t_k - \varepsilon_k, t_k + \varepsilon_k] \cap [0,T]} \Big| z(\tau) - z(t_k-) \mathds{1}_{(-\infty, t_k)}(\tau)\Big | \psi_k(s)
\right ).
\end{aligned}
\end{equation*} 
Due to the inequalities in  \cref{lemma:polyprep1}\ref{lemma:polyprep1:i} and  \cref{lemma:polyprep2}\ref{lemma:polyprep2:i},
our choice $\xi = 1/j$, and the properties of $\psi_k$,
the last estimate yields
\[
\sup_{s \in [0,T]}\left |
z(s) - z_j(s)
\right |
 \leq 
\sup_{s \in [0,T]} 
\left ( 
\sum_{k \colon z(t_k)= z(t_k-)} 
\frac{\psi_k(s)}{j}  +
\sum_{k \colon z(t_k) \neq z(t_k-)} 
\frac{\psi_k(s)}{j}   
\right )
=
\frac{1}{j}.
\]
This shows that the sequence $\{z_j\}$ indeed converges uniformly to $z$ for $j \to \infty$.
That we have $\|z_{i,j}\|_\infty \leq \|z\|_\infty$ and $\|z_{j}\|_\infty \leq \|z\|_\infty$
follows immediately from our construction and the properties of 
$\psi_k$ and $\varphi$. This completes the proof. 
\end{proof}

Note that, to be able to 
establish that
 $\smash{\KK_{C^\infty}^{\rad,\crit}(y,u)}$
is dense in $\smash{\KK_{G_r}^{\red,\crit}(y,u)}$,
one necessarily has to consider a type of convergence weaker than 
uniform convergence since otherwise it is not possible to 
leave the space $C[0,T] \supset \smash{\KK_{C^\infty}^{\rad,\crit}(y,u)}$. This is a major 
difference between the 
temporal polyhedricity result in
\cref{theorem:tempoly} and the classical
notion of polyhedricity for the 
elliptic obstacle problem in 
\eqref{eq:optstaclecontrol} which yields the density 
of the set of critical radial directions 
$ K_{\rad}(y) \cap (u + \Delta y)^\perp$
in the critical cone 
$K_{\tan}(y) \cap (u + \Delta y)^\perp$
in $(H_0^1(\Omega), \|\cdot\|_{H_0^1})$ 
and thus in the topology that 
is natural for the underlying variational inequality. 
For \eqref{eq:V},
this natural choice of the topology would be that 
of uniform convergence as the Lipschitz estimate \eqref{eq;LipschitzSinfty} shows.

Before we apply \cref{theorem:tempoly} to derive strong stationarity conditions for 
\eqref{eq:P}, we prove a further auxiliary result.

\begin{lemma}%
\label{lemma:polarconeatoms}%
Suppose that $t \in [0,T]$ is given
and let $c \in \R$ be an element of the polar cone $K^\ptw_{\crit}(y,u)(t)^\circ$, i.e., 
the set 
\begin{equation}
\label{eq:ptwnormalconedef}
K^\ptw_{\crit}(y,u)(t)^\circ 
:=
\begin{cases}
\{0\} &\text{ if } t \in I(y),
\\
[0, \infty) & \text{ if } t \in B_+(y,u),
\\
(-\infty, 0]& \text{ if } t \in B_-(y,u),
\\
\R &\text{ if } t \in  A(y,u).
\end{cases}
\end{equation}
Then there exists a sequence $\{h_i\} \subset C^\infty[0,T]$ such that 
the following holds:
\begin{equation*}
\begin{gathered}
\|h_i\|_\infty \leq |c| \text{ and } \|\S'(u;h_i)\|_\infty \leq 2|c|~\forall i \in \mathbb{N},
\\
\S'(u;h_i)_+ \to 0 \text{ pointwise in } [0,T] \text{ for } i \to \infty,
\\
h_i  \to c \mathds{1}_{[t,T]} \text{ pointwise in } [0,T] \text{ for } i \to \infty.
\end{gathered}
\end{equation*}
\end{lemma}

\begin{proof}
If $t \in I(y)$, then we necessarily have $c=0$ and we can simply choose the sequence $h_i = 0$ for all $i$.
If $t=0$, then 
the sequence defined by $h_i = c$ for all $i$ satisfies all assertions because 
$\S'(u;c\mathds{1}_{[0,T]}) = 
\S'(u;c\mathds{1}_{[0,T]})_+ = 0$ by 
\cref{th:dirdiffVI} in view of \cref{eq:integralofconstant}.
We may thus assume that 
\[
0 < t \in B(y,u)\cup  A(y,u).
\]
Consider an arbitrary but fixed function $\varphi$ with the following properties 
\[
\varphi \in C^\infty(\R),
\quad \varphi(s) = 0 ~\forall s \in (-\infty,-1],
\quad \varphi(s) = 1 ~\forall s \in [0, \infty),
\quad \varphi'(s) \geq 0~\forall s \in \R.
\]
We define $\{h_i\}$ via 
\[
h_i(s) := c \varphi\left ( \frac{s - t}{1/i}\right )\quad \forall s \in [0,T]\quad \forall i \in \mathbb{N}. 
\]
This sequence clearly satisfies $\{h_i\} \subset C^\infty[0,T]$,
$h_i(s)  \to c \mathds{1}_{[t,T]}(s)$ for all $s \in [0,T]$ and $i \to \infty$,
and $\|h_i\|_\infty = |c|$ for all $i$.
 Due to the Lipschitz estimate
 \eqref{eq;LipschitzSinfty}, 
 this also implies that $\|\S'(u;h_i) \|_\infty \leq 2|c|$
holds for all $i$.

It remains to establish the pointwise convergence of $ \S'(u;h_i)_+$ to zero.
For this to hold, it suffices to prove that 
$\eta_i := \S'(u;h_i)_+$ satisfies
$\eta_i = 0$ on $[0,t-1/i] \cup [t,T]$ for all $i$ with $1/i < t$. 
That $\eta_i$ vanishes on $[0,t-1/i]$
follows easily from 
the fact that $h_i$ is zero 
on $[0, t-1/i]$,
\eqref{partialIntegration},
and 
\eqref{eq:dirdiffVI+s}
with $z = 0$, $z = 2 \eta$, and 
$0 < s \leq t - 1/i$.
Next, we prove that $\eta_i(t) = 0$ by distinguishing three cases.

Case 1: $t\in  A(y,u)$. In this case, we have $\eta_i(t) \in K^\ptw_{\crit}(y,u)(t) = \{0\}$. 

Case 2: $t\in B_+(y,u)$. In this case, 
we have
$\eta_i(t) \in K^\ptw_{\crit}(y,u)(t) = (-\infty, 0]$,
it holds $c \in [0, \infty)$, and 
 $h_i$ is nondecreasing on $[0,T]$. By  \cref{lemma:monotonicity},
 this yields 
 $\S(u+\alpha h_i) \geq \S(u)$ 
 in $[0, T]$ for all $\alpha > 0$ and all $i \in \mathbb{N}$.
 Hence, $\S'(u;h_i) \geq 0$ in $[0,T]$
 and, consequently, 
 $\eta_i = \S'(u;h_i)_+ \geq 0$ in $[0,T]$.
 It follows that $\eta_i(t) = 0$.

Case 3: $t\in B_-(y,u)$. In this case, it holds 
$\eta_i(t) \in K^\ptw_{\crit}(y,u)(t) = [0, \infty)$
and $c \in (-\infty, 0]$, and we can proceed 
completely analogously to Case 2 
(with reversed signs) to obtain that 
$\eta_i(t) = 0$.

It remains to prove that $\eta_i = 0$ on $(t,T]$ if $t < T$.
Let $\hat{\eta}_i :=  \mathds{1}_{[0,t]}\eta_i =  \mathds{1}_{[0,t)}\eta_i$.
By the definition of $h_i$, the function $\hat{\eta}_i - h_i$ has the constant value $-c$ on $[t,T]$.
Using \cref{eq:decomposeInterval} combined with \cref{eq:integralofconstant}, we obtain that,
for all $ z \in G\left ([0, T]; K^\ptw_{\crit}(y,u) \right)$, we have
\begin{align*}
\int_0^T (z  - \hat{\eta}_i)\dd(\hat{\eta}_i - h_i) 
&= \int_0^t (z  - \hat{\eta}_i)\dd(\hat{\eta}_i - h_i)
+ \int_t^T (z  - \hat{\eta}_i)\dd(\hat{\eta}_i - h_i)
\\ &=
\int_0^t (z  - \hat{\eta}_i)\dd(\hat{\eta}_i - h_i)
\\
&= \int_0^t (z  - \eta_i)\dd(\eta_i - h_i)
\ge 0,
\end{align*}
where the last inequality holds 
by \cref{cor:dirdiffVI+}. 
Since $\hat{\eta}_i(s) \in  K^\ptw_{\crit}(y,u)(s)$ for all $s\in  [0,T]$,
we conclude that $\hat{\eta}_i$ solves \cref{eq:dirdiffVI+} for $h = h_i$. As $\eta_i$ is the unique solution of
\cref{eq:dirdiffVI+}, we must have $\hat{\eta}_i = \eta_i$.
Thus, $\eta_i = 0$ on $(t,T]$ and the proof is complete.
\end{proof}

\section{Strong stationarity condition}\label{sec:7}

We are now in the position to prove 
the strong stationarity system \eqref{eq:strongstatsys-2}.

\begin{theorem}[strong stationarity]%
\label{th:main}%
Consider the situation in \cref{ass:standing} and suppose that 
$\bar u \in \Uad$ is a control with state $\bar y := \S(\bar u)$ 
such that the set $\R_+(\Uad - \bar u)$ is dense in $U$.
Then $\bar u$ 
is a Bouligand stationary point of \eqref{eq:P}, i.e., 
satisfies \eqref{eq:Bouligand} if and only if there exist
an adjoint state $\bar p \in BV[0,T]$ and a multiplier $\bar \mu \in G_r[0,T]^*$
such that the following system is satisfied:
\begin{equation}
\label{eq:strongstatsys}
\begin{gathered}
\bar p(0) = \bar p(T) = 0,
\quad 
\bar p(t) = \bar p(t-)~\forall t \in [0,T),
\\
\bar p(t-) \in K^\ptw_{\crit}(\bar y, \bar u)(t)~\forall t \in [0,T],
\\
\left \langle \bar \mu, z \right \rangle_{G_r} \geq 0\quad \forall z \in \KK_{G_r}^{\red,\crit}(\bar y, \bar u)  ,
\\
\int_0^T h \dd \bar p = \left \langle \partial_3 \JJ(\bar y, \bar y(T), \bar u), h \right \rangle_{U} ~\forall h \in U,
\\
-\int_0^T z \dd \bar p  
=
\left \langle \partial_1 \JJ(\bar y, \bar y(T), \bar u), z\right \rangle_{L^\infty}
+
\partial_2 \JJ(\bar y, \bar y(T), \bar u)z(T)
-
\left \langle \bar \mu, z \right \rangle_{G_r}
\\
\hspace{8.5cm}\forall z \in G_r[0,T]. 
\end{gathered}
\end{equation}
\end{theorem}

\begin{proof}
We begin with the proof of the implication ``\eqref{eq:Bouligand} $\Rightarrow$ \eqref{eq:strongstatsys}'':
Suppose that a control $\bar u \in \Uad$ with state $\bar y := \S(\bar u)$ 
is given such that the set $\R_+(\Uad - \bar u)$ is dense in $U$
and such that \eqref{eq:Bouligand} holds. Then 
it follows from \eqref{eq:Bouligand},
the fact that \eqref{eq;LipschitzSinfty} implies that 
$
\|\S'(\bar u;h_1) - \S'(\bar u; h_2)\|_\infty \leq 2\|h_1 - h_2\|_\infty
$
holds for all $h_1, h_2 \in CBV[0,T]$,
the inclusion $U \subset CBV[0,T]$,
and the continuity 
of the embedding $U \hookrightarrow C[0,T]$ that
\begin{equation}
\label{eq:Bouligandstat2}
\begin{aligned}
\left \langle \partial_1 \JJ(\bar y, \bar y(T), \bar u), \S'(\bar u; h) \right \rangle_{L^\infty} &+
\partial_2 \JJ(\bar y, \bar y(T), \bar u)\S'(\bar u; h)(T)
\\
&+
\left \langle \partial_3 \JJ(\bar y, \bar y(T), \bar u), h \right \rangle_{U}
\geq 
0
\qquad 
\forall h \in U.
\end{aligned}
\end{equation}
Again due to \eqref{eq;LipschitzSinfty}
and since $-h \in U$ holds for all $h \in U$, \eqref{eq:Bouligandstat2} yields
\[
\left | \left \langle \partial_3 \JJ(\bar y, \bar y(T), \bar u), h \right \rangle_{U} \right |
\leq 2 \left ( \left \| \partial_1 \JJ(\bar y, \bar y(T), \bar u)\right \|_{L^1} 
 + \left | \partial_2 \JJ(\bar y, \bar y(T), \bar u) \right |
\right )\|h\|_{\infty}
\]
for all $h \in U$. In combination with the Hahn-Banach theorem, this shows that 
the linear functional $U \ni h \mapsto \left \langle \partial_3 \JJ(\bar y, \bar y(T), \bar u), h \right \rangle_{U}  \in \R$
can be extended to an element of 
the dual space $C[0,T]^*$.
In view of the classical Riesz representation theorem 
(see, e.g., \cite[section 8.1]{Monteiro2019})
and \cref{lemma:intgzeroexcept}, this means that there exists a function
$\bar p \in BV[0,T]$ satisfying $\bar p(t) = \bar p(t-)$ for all $t \in (0,T)$, $\bar p(T) = 0$, and 
\[
\int_0^T h \dd \bar p = \left \langle \partial_3 \JJ(\bar y, \bar y(T), \bar u), h \right \rangle_{U} ~\forall 
h \in U.
\]
Since $\partial_1 \JJ(\bar y, \bar y(T), \bar u) \in L^1(0,T)$, 
$\S'(\bar u; h) = \S'(\bar u; h)_+$ a.e.\ in $(0,T)$ by \cite[Theorem 2.3.2]{Monteiro2019},
and $\S'(\bar u; h)(T) 
= \S'(\bar u; h)(T+)
= \S'(\bar u; h)_+(T)$ by definition, we may now 
rewrite \eqref{eq:Bouligandstat2} as follows:
\begin{equation}
\label{eq:randomeq263536}
\begin{aligned}
&\int_0^T h \dd \bar p 
+
\left \langle \partial_1 \JJ(\bar y, \bar y(T), \bar u), \S'(\bar u; h)_+ \right \rangle_{L^\infty}
+
\partial_2 \JJ(\bar y, \bar y(T), \bar u)\S'(\bar u; h)_+(T) \geq 0
\\
&\hspace{10.5cm}
\forall h \in U.
\end{aligned}
\end{equation}
Note that, again due to
the Lipschitz estimate 
$\|\S'(\bar u;h_1) - \S'(\bar u; h_2)\|_\infty \leq 2\|h_1 - h_2\|_\infty$
for $h_1, h_2 \in CBV[0,T]$ and since $U$ is dense in $C[0,T]$,
\eqref{eq:randomeq263536} remains valid when the test space $U$
is replaced by $CBV[0,T]$.
We define $\bar \mu \in G_r[0,T]^*$ via 
\[
\left \langle \bar \mu, z \right \rangle_{G_r} := 
\int_0^T z \dd \bar p 
+
\left \langle \partial_1 \JJ(\bar y, \bar y(T), \bar u), z \right \rangle_{L^\infty}
+
\partial_2 \JJ(\bar y, \bar y(T), \bar u)z(T)
\quad 
\forall z \in G_r[0,T].
\]
Then the last line in \eqref{eq:strongstatsys} holds, and it follows from \eqref{eq:randomeq263536} with test space $CBV[0,T]$ and 
\cref{lemma:dirdifcritrad} that 
\[
\int_0^T z \dd \bar p 
+
\left \langle \partial_1 \JJ(\bar y, \bar y(T), \bar u), z \right \rangle_{L^\infty}
+
\partial_2 \JJ(\bar y, \bar y(T), \bar u)z(T) \geq 0
\quad \forall  z \in \KK_{C^\infty}^{\rad,\crit}(\bar y, \bar u).
\]
Due to \cref{theorem:tempoly} and the bounded convergence theorem (\cref{th:boundedconv}),
we can extend the last inequality to the set $\KK_{G_r}^{\red,\crit}(\bar y, \bar u)$ by approximation, i.e., 
we have 
\[
\int_0^T z \dd \bar p 
+
\left \langle \partial_1 \JJ(\bar y, \bar y(T), \bar u), z \right \rangle_{L^\infty}
+
\partial_2 \JJ(\bar y, \bar y(T), \bar u)z(T) = \left \langle \bar \mu, z \right \rangle_{G_r} \geq 0
\]
for all $z \in \KK_{G_r}^{\red,\crit}(\bar y, \bar u)$.
This proves the third line in \cref{eq:strongstatsys}.
It remains to establish the pointwise properties of $\bar p$ in \cref{eq:strongstatsys}. 
To this end, we again use that \cref{cor:dirdiffVI+} and \cref{eq:integralofconstant} imply that
$\S'(\bar u; c\mathds{1}_{[0,T]})_+ = 0$ holds
for all $c \in \R$.
By \eqref{eq:randomeq263536} with test 
space $CBV[0,T]$, this yields
\begin{equation*}
\begin{aligned}
0 \leq c\int_0^T  \dd \bar p = c\left ( \bar p(T) - \bar p(0)\right )\quad \forall c \in \R.
\end{aligned}
\end{equation*}
Thus, $\bar p(0) = \bar p(T)$.
Since $\bar p(T) = 0$
and $\bar p(t) = \bar p(t-)$ for all $t \in (0,T)$, and since 
$\bar p(0) = \bar p(0-)$ holds by definition, 
this establishes 
the first line of \cref{eq:strongstatsys}.
Next, by invoking \cref{lemma:polarconeatoms},
by 
setting $h = h_i$ in \eqref{eq:randomeq263536} 
with test space $CBV[0,T]$,
and by passing
to the limit $i\to\infty$ by means of \cref{th:boundedconv} and the dominated convergence theorem,
we obtain that, for every $t \in [0,T]$ and every $c \in K^\ptw_{\crit}(\bar y,\bar u)(t)^\circ$, we have 
\begin{equation}
\label{eq:randomeq2635}
0 \leq 
\int_0^T c \mathds{1}_{[t,T]}  \dd \bar p 
= c \left ( \bar p(T) - \bar p(t-) \right )
= - c \bar p(t-).
\end{equation}
Here, the last two equations follow from \cite[Lemma 6.3.3]{Monteiro2019} 
and the identity
$\bar p(T) = 0$.
By using the definition \eqref{eq:ptwnormalconedef} of the 
polar cone $K^\ptw_{\crit}(\bar y, \bar u)(t)^\circ$
in \eqref{eq:randomeq2635}, one readily obtains that 
$\bar p(t-) \in K^\ptw_{\crit}(\bar y, \bar u) (t)$
holds for all $t \in [0,T]$.
This establishes the second line in \eqref{eq:strongstatsys}
and proves, in combination with the previous steps, that the
strong stationarity system \eqref{eq:strongstatsys} 
is indeed a necessary condition for Bouligand stationarity. 

Next, we prove the implication 
``\eqref{eq:strongstatsys} $\Rightarrow$ \eqref{eq:Bouligand}''.
Suppose that $\bar u \in \Uad$ is a control 
with state $\bar y := \S(\bar u)$
such that there exist $\bar p \in BV[0,T]$ and
$\bar \mu \in G_r[0,T]^*$  satisfying \eqref{eq:strongstatsys}. 
Assume further that a direction $h \in U$ is given 
and define $\eta := \S'(\bar u;h)_+$. 
Then it follows from 
the properties of $\bar p$,
\eqref{eq:dirdiffVI+} with $z := \eta + \bar p$,
and the integration by parts formula 
for the Kurzweil-Stieltjes integral
\cite[Theorem 6.4.2]{Monteiro2019} 
that
\begin{equation*}
\begin{aligned}
0  \leq \int_0^T \bar p\dd(\eta- h)
&=
\int_0^T (h - \eta) \dd \bar p
+
\bar p(T)(\eta- h)(T) - \bar p(0)(\eta- h)(0)
\\
&\qquad +
\sum_{t \in [0,T]} \left ( \bar p(t) - \bar p(t-)\right ) \left ( (\eta- h)(t)  - (\eta- h)(t-) \right ) 
\\
&\qquad - \sum_{t \in [0,T]} \left ( \bar p(t) - \bar p(t+)\right )  \left ( (\eta- h)(t)  - (\eta- h)(t+) \right ).
\end{aligned}
\end{equation*}
Due to the identities $\bar p(0) = \bar p(T) = 0$ and $\eta(0) = \eta(0-) = 0$ and due to the left- and right-continuity 
properties of 
$\bar p$, $h$, and $\eta = \S'(\bar u;h)_+$, the last estimate simplifies to 
\[
0 \leq 
\int_0^T (h - \eta) \dd \bar p
- \bar p(T-) \left ( \eta(T)  - \eta(T-) \right ).
\]
Note that \eqref{eq:dirdiffjumps1},
$\bar p(T-) \in K^\ptw_{\crit}(\bar y,\bar u) (T)$,
and the convention $\eta(T) = \eta(T+)$
imply that 
$\bar p(T-) ( \eta(T)  - \eta(T-)   )
=
\bar p(T-) ( \eta(T+)  - \eta(T-)   )
\geq 0$ holds.
We thus obtain 
\[
0 \leq 
\int_0^T (h - \eta) \dd \bar p
=
\int_0^T h \dd \bar p
-
\int_0^T \eta \dd \bar p,
\]
and, by the last three lines of \cref{eq:strongstatsys} and the properties of $\eta$, 
\begin{equation*}
\begin{aligned}
0 &\leq 
\left \langle \partial_3 \JJ(\bar y, \bar y(T), \bar u), h \right \rangle_{U}
+ 
\left \langle \partial_1 \JJ(\bar y, \bar y(T), \bar u), \eta \right \rangle_{L^\infty}
+
\partial_2 \JJ(\bar y, \bar y(T), \bar u)\eta(T)
-
\left \langle \bar \mu, \eta \right \rangle_{G_r}
\\
&\leq 
\left \langle \partial_3 \JJ(\bar y, \bar y(T), \bar u), h \right \rangle_{U}
+ 
\left \langle \partial_1 \JJ(\bar y, \bar y(T), \bar u), \eta\right \rangle_{L^\infty}
+
\partial_2 \JJ(\bar y, \bar y(T), \bar u)\eta(T).
\end{aligned}
\end{equation*}
If we now exploit that 
$\partial_1 \JJ(\bar y, \bar y(T), \bar u) \in L^1(0,T)$,
that $\S'(\bar u; h)(T) = \S'(\bar u; h)(T+)$,
and that $\eta = \S'(\bar u;h)$ a.e.,
then \eqref{eq:Bouligand} follows. This completes the proof.
\end{proof}

Note that, in the case $T \in I(\bar y)$, there exists 
$m > 0$ such that the function
$z(t) := c \mathds{1}_{[T -\varepsilon, T]}(t) (t - T + \varepsilon)/\varepsilon$ is an element of
\smash{$\KK_{G_r}^{\red,\crit}(\bar y, \bar u)$} 
for all $c \in \R$ and all $0 < \varepsilon < m$.
For such a function $z$, the third line of \eqref{eq:strongstatsys}
becomes $\left \langle \bar \mu, z \right \rangle_{G_r} = 0$.
Using this
in the fifth line of \eqref{eq:strongstatsys}
and 
subsequently passing to the limit $\varepsilon \to 0^+$
by means of \cref{th:boundedconv}
yields,
due to the $L^1$-regularity of $\partial_1 \JJ(\bar y, \bar y(T), \bar u)$
and \cref{eq:singlepointmass},
that
\[
-c \int_0^T \mathds{1}_{\{T\}} \dd \bar p  
=
-c (\bar p(T) - \bar p(T-))
=
\partial_2 \JJ(\bar y, \bar y(T), \bar u)c\qquad \forall c \in \R.
\]
Thus,
$\bar p(T-) - \bar p(T) =  \partial_2 \JJ(\bar y, \bar y(T), \bar u)$
and we obtain that the partial derivative 
$\partial_2 \JJ(\bar y, \bar y(T), \bar u) \in \R$ affects the jump
of 
$\bar p$ at $T$, as mentioned in \cref{sec:1}.
We remark that, by redefining $\bar p$, this implicit jump condition 
on the adjoint state in  \eqref{eq:strongstatsys} can 
also be transformed into a condition on the function value at $T$.
Indeed, by introducing the modified adjoint state 
$\bar q := \bar p + \partial_2 \JJ(\bar y, \bar y(T), \bar u)\mathds{1}_{\{T\}} \in BV[0,T]$, 
by using the integration by parts formula in \cite[Theorem 6.4.2]{Monteiro2019}
in the fourth line of \eqref{eq:strongstatsys}, 
and by employing \eqref{eq:singlepointmass} and \cite[Lemma 6.3.2]{Monteiro2019}, one easily checks that 
the strong stationarity system in \cref{th:main} can also be formulated as follows:
\begin{equation*}
\begin{gathered}
\bar q(0) =  0,
\quad
\bar q(T) = \partial_2 \JJ(\bar y, \bar y(T), \bar u),
\quad 
\bar q(t) = \bar q(t-)~\forall t \in [0,T),
\\
\bar q(t-) \in K^\ptw_{\crit}(\bar y, \bar u)(t)~\forall t \in [0,T],
\\
\left \langle \bar \mu, z \right \rangle_{G_r} \geq 0\quad \forall z \in \KK_{G_r}^{\red,\crit}(\bar y, \bar u)  ,
\\
- \int_0^T \bar q \dd h = \left \langle \partial_3 \JJ(\bar y, \bar y(T), \bar u), h \right \rangle_{U} ~\forall h \in U,
\\
-\int_0^T z \dd \bar q
=
\left \langle \partial_1 \JJ(\bar y, \bar y(T), \bar u), z\right \rangle_{L^\infty}
-
\left \langle \bar \mu, z \right \rangle_{G_r}\quad \forall z \in G_r[0,T]. 
\end{gathered}
\end{equation*}

Regarding the assumption that the set $\R_+(\Uad - \bar u)$
is dense in $U$, we would like to point out that 
this so-called  ``ample control'' condition 
in \cref{th:main} is rather restrictive and 
rarely satisfied if $\Uad \neq U$. Using techniques
from \cite{Wachmuth2014}, it might be possible to establish a strong stationarity system for \eqref{eq:P} 
also under weaker assumptions on the control constraints.
We leave this topic for future research.

\appendix
\section{Results on the Kurzweil-Stieltjes integral}\label{sec:appendix}%
Let $a,b \in \R$ with $a < b$ be given.
For $f, g \in G[a,b]$, the Kurzweil-Stieltjes integral 
with \emph{integrand} $f$ and 
\emph{integrator} $g$ 
exists 
if at least one of the functions $f$ and $g$ 
has bounded variation,
see \cite[Theorem 6.3.11]{Monteiro2019}.
In this case, it yields a real number which we denote by
\[
\int_a^b f\dd g \qquad \text{or} \qquad \int_a^b f(t)\dd g(t).
\]
The Kurzweil-Stieltjes integral 
coincides with the Riemann-Stieltjes integral whenever the latter exists, see \cite[Theorem 6.2.12]{Monteiro2019}.
This holds in particular if $f\in C[a,b]$ and $g\in BV[a,b]$, see \cite[Theorem 5.6.1]{Monteiro2019}.
If $c\in\R$ is interpreted as a constant function, then
it holds
\begin{equation}\label{eq:integralofconstant}
\int_a^b c\dd g = c(g(b) - g(a))
 \qquad \text{and} \qquad
 \int_a^b f\dd c = 0
\end{equation}
for all $f, g \in G[a,b]$, 
see \cite[Remark 6.3.1]{Monteiro2019}.

The Kurzweil-Stieltjes integral is linear w.r.t.\
the integrand $f$ and w.r.t.\
the integrator $g$,
see \cite[Theorem 6.2.7]{Monteiro2019}.
Further,
for all $c\in (a,b)$, it holds
\begin{equation}\label{eq:decomposeInterval}
\int_a^b f\dd g = \int_a^c f\dd g + \int_c^b f\dd g
\end{equation}
provided the first integral exists, 
see \cite[Theorems 6.2.9, 6.2.10]{Monteiro2019}. 
For $t\in [a,b]$ and $g \in G[a,b]$, we have (see \cite[Lemma 6.3.3]{Monteiro2019})
\begin{equation}\label{eq:singlepointmass}
\int_a^b \mathds{1}_{\{t\}}\dd g = g(t+) - g(t-)
\end{equation}
with the conventions $g(b+) := g(b)$ and 
$g(a-) := g(a)$.
In particular, the integral in 
\eqref{eq:singlepointmass} equals zero if $g$ is continuous at $t$.

\begin{lemma}\label{lemma:subintervals}
Let $f\in G[a,b]$, $g\in BV_r[a,b]$, 
$a\le s < \tau\le b$, and 
$J := (s,\tau]$. 
Then
\begin{equation}\label{eq:subintervals}
\int_s^\tau f\dd g = \int_a^b \mathds{1}_J f\dd g.
\end{equation}
If $g\in CBV[a,b]$, then
\cref{eq:subintervals} 
is also true 
for $J = [s,\tau]$, $J = (s,\tau)$, and $J = [s,\tau)$.
\end{lemma}

\begin{proof}
This is a special case of 
\cite[Theorem 6.9.7]{Monteiro2019}.
\end{proof}

\begin{lemma}\label{lemma:intgzeroexcept}
Let $g\in BV[a,b]$ be given such that
$g(a) = g(b) = 0$ holds
and such that the set $\{t\in [a,b]\colon g(t)\neq 0\}$ is finite or countably infinite. Then 
\[
\int_a^b f\dd g = 0 \qquad \forall f\in G[a,b].
\]
\end{lemma}

\begin{proof}
This is a special case of 
\cite[Lemma 6.3.15]{Monteiro2019}.
\end{proof}

\begin{proposition}\label{prop:partialIntegration}
Let $g\in BV_r[a,b]$. Then
\begin{equation}\label{partialIntegration}
\int_a^b g\dd g = \frac{1}{2}(g(b)^2 - g(a)^2) 
+ \frac{1}{2} \sum_{t\in [a,b]} (g(t) - g(t-))^2.
\end{equation}
\end{proposition}

\begin{proof}
This is a special case of \cite[Corollary 2.12]{Krejci2003} or of \cite[Corollary 1.13]{KrejciLiero2009}. 
\end{proof}

\begin{theorem}[bounded convergence theorem]\label{th:boundedconv}%
Let $g\in BV[a,b]$, $f_n\in G[a,b]$ with $\sup_n \|f_n\|_\infty < \infty$ and $f_n\to f$ pointwise in $[a,b]$ be given. Then 
the integral 
$
\int_a^b f\dd g
$
exists and it holds 
\[
\lim_{n\to\infty} \int_a^b f_n\dd g = \int_a^b f\dd g.
\]
\end{theorem}

\begin{proof}
This is a special case of \cite[Theorem 6.8.13]{Monteiro2019}.
\end{proof}

\bibliographystyle{siamplain}
\bibliography{references}

\end{document}